\documentclass[final,11pt,fleqn,leqno]{siamltex}
\usepackage[bitstream-charter]{mathdesign}
\usepackage{amsfonts}
\usepackage{amsmath}
\usepackage{algorithm}
\usepackage{algpseudocode}
\usepackage{url}
\usepackage[pdftex]{geometry,graphicx,color}
\def\PP{{\mathbb P}}
\def\CC{{\mathbb C}}
\def\mult{\operatorname{mult}}
\def\Id{I}

\newcounter{myalg}
  {
    \needspace{2\baselineskip}
    \noindent \rule{\linewidth}{1pt} \endgraf
    \refstepcounter{myalg}
    \centering \textsc{Algorithm}~\themyalg%
    \ifthenelse{\isempty{#1}}{}{:\ #1}
  }{
  \noindent \rule{\linewidth}{1pt}
  }%

\newtheorem{example}[theorem]{Example}

\newtheorem{remark}[theorem]{Remark}

\title{Simple determinantal representations of up to quintic
bivariate polynomials}

\author{
Anita Buckley\thanks{%
Department of Mathematics, University of Ljubljana, Jadranska 19,
SI-1000 Ljubljana, Slovenia, {\tt anita.buckley@fmf.uni-lj.si}.}
\and
Bor~Plestenjak\thanks{%
IMFM and Department of Mathematics, University of Ljubljana, Jadranska 19,
SI-1000 Ljubljana, Slovenia, {\tt bor.plestenjak@fmf.uni-lj.si}. This author
was partially supported by the Slovenian Research Agency (P1-0294).}
}

\begin{document}
\maketitle

\begin{abstract}
For bivariate polynomials of degree $n\le 5$ we give
fast numerical constructions of determinantal representations
with $n\times n$ matrices. Unlike some other available
constructions, our approach returns matrices of the smallest
possible size $n\times n$
for all polynomials of degree $n$ and does not require any symbolic
computation. We can apply these linearizations
to numerically compute the roots of a system of two bivariate polynomials by
using numerical methods for two-parameter eigenvalue problems.

\end{abstract}

\begin{keywords}
bivariate polynomial, determinantal representation
\end{keywords}

\begin{AMS}
65F15, 65H04, 65F50, 13P15, 14M12, 14Q05.
\end{AMS}

\section{Introduction}
We say that matrices $A$, $B$, and $C$ form a \textit{determinantal representation} of a bivariate polynomial $p$
if
$$p(x,y)=\det(xA+yB+C).$$
Dixon showed in 1902 \cite{Dixon} that for each
bivariate polynomial of degree $n$ there exists
a determinantal representation with $n\times n$ symmetric matrices.
Even when the matrices are allowed to be nonsymmetric, such
representations are difficult to construct, also for polynomials
of small degrees. In this paper we introduce simple constructions
that can be applied to all bivariate polynomials of degree 5 or less.

Recently, Plestenjak and Hochstenbach
applied determinantal representations in \cite{BorMichiel}
to numerically find roots of a system of two bivariate polynomials
using  numerical methods
for singular two-parameter eigenvalue problems. To make this
approach efficient, one needs determinantal representations with
matrices as small as possible that
can be constructed efficienty. By Dixon, the optimal size is $n\times n$
for a bivariate polynomial of degree $n$ but at present,
no efficient construction for $n\times n$ representations is known
that could be applied to all polynomials.

The above requirements
are most closely met by a recent algorithm in \cite{Bornxn} that, using only simple numerical
 computations, returns a determinantal representation
with $n\times n$ matrices of a square-free bivariate polynomial of degree $n$ and a representation with
$(2n-2)\times (2n-2)$ matrices of a non square-free polynomial.
It is important that it does not require any symbolic
computation, which usually is the bottleneck for this kind of
algorithms.

The algorithm in \cite{Plaumann} gives a determinantal representation
with $n\times n$ matrices for polynomials that satisfy the real zero
condition, however it is computationally too expensive and thus not suitable
as a building block of a root finding software for bivariate polynomials.
Also, we need determinantal representations for all bivariate
polynomials of degree$~n$.

In \cite{BorMichiel} two constructions of determinantal representations
are presented, which can both be constructed fast with little numerical
computation. For generic bivariate polynomials of degrees $3$, $4$, and $5$
the construction in \cite{BorMichiel} returns determinantal
representations with matrices of sizes $3\times 3$, $5\times 5$, and
$8\times 8$, respectively. In addition, the algorithm can fail for
certain cubic and quartic polynomials, in which case the size
of the matrices increases by one.

While the above constructions do not give $n\times n$ representations
 for all bivariate polynomials of degree $n$, we fill the missing gaps for degrees up to 5. For  every bivariate polynomial of degree $n\le 5$ we present
a simple numerical algorithm that returns a representation with $n\times n$ matrices.

The paper is organized as follows. In Section~\ref{sec:planecurves}
we review some basic results on algebraic curves in the complex
projective plane, including a thorough description of pencils of conics. In Section~\ref{sec:ncurve} we introduce a reduction technique to polynomials of lower degrees which is used later as the main tool in our constructions.
In Sections~\ref{sec:quadratic} through \ref{sec:quintic} we give $n\times n$ determinantal representations of bivariate
polynomials of degrees $2,3,4$, and $5$, respectively. We show in Section~\ref{sec:sextic} that this approach can not be applied to polynomials of degree $6$. In Section~\ref{sec:numalg} we
join the methods from the previous sections in an algorithm for
determinantal representations. Some numerical results are listed in Section~\ref{sec:numalg} and we end with conclusions.

\section{Curves in complex projective plane}\label{sec:planecurves}

Let $p\in \CC[x,y]$ be a bivariate polynomial of degree $n$.
In the language of algebraic geometry, its set of zeros $\mathcal{C}=\left\{(x,y)\in\CC^2\, :\, p(x,y)=0 \right\}$ defines an \textit{affine algebraic curve}.  By abuse of notation we often say curve $\{p(x,y)=0\}$ or even shorter curve $p$.
A \textit{determinantal representation} of $p$ or of $\mathcal{C}$ is an expression
$$p(x,y)=\det(xA+yB+C),$$
where $A,B,C$ are $m\times m$ matrices with $m\geq n$.

When $m=n$ it is natural to homogenize matrices $A,B,C$ by introducing a new variable $z$ into the determinantal representation $xA+yB+zC$. Then
$$\det(xA+yB+zC)=z^n\, p(x/z,y/z)$$
is a homogeneous polynomial, which for the sake of a shorter notation, we denote by $p(x,y,z)$. Its set of zeros
$$\mathcal{C}=\left\{(x,y,z)\in\CC\PP^2\, :\, p(x,y,z)=0 \right\}$$
defines a \textit{projective curve} in the complex projective plane. Recall that by definition
$$\CC\PP^2=\left\{(x,y,z)\in \left\{\CC^3-(0,0,0)\right\} /_{\sim} \ :\  (x,y,z)\sim\lambda (x,y,z) \mbox{ for all } 0\neq \lambda\in \CC \right\}.$$
By analogy to the affine case, we often say projective curve  $\{p(x,y,z)=0\}$ or  $p$ for the zero locus of the homogeneous polynomial $p(x,y,z)$.

Since $ (x/z,y/z,1)=(x,y,z)$ in $\CC\PP^2$ it is easy to transit between the afine and projective curves. Indeed, given a homogeneous polynomial $p(x,y,z)$, the zero locus $p(x,y,1)=0$ defines an afine curve in $\CC^2$. Conversely, a
bivariate polynomial $p(x,y)$ induces a homogeneous form $z^n\, p(x/z,y/z)=p(x,y,z)$ whose set of zeros is a projective curve.


We will extensively avail of
 B\'ezout's theorem discovered in 1765, which counts the number of points in the intersection of two plane curves with no common components.
 \begin{theorem}[B\'ezout's theorem~\cite{fischer}] For algebraic curves $\mathcal{C}_1, \mathcal{C}_2\subset \CC\PP^2$ that have no common component,
 $$\sum_{T\in \mathcal{C}_1\cap \mathcal{C}_2 } \mult_T(\mathcal{C}_1\cap \mathcal{C}_2)=
 \deg  \mathcal{C}_1 \cdot  \deg \mathcal{C}_2.$$
 \end{theorem}

The multiplicity of intersection is invariant under projective transformations. A \textit{projective transformation} is a bijection
$P: \CC\PP^2 \rightarrow \CC\PP^2$ defined as
\begin{equation}\label{eq:transt}
\left[\begin{matrix}x \cr y \cr z\end{matrix}\right] \mapsto
\left[\begin{matrix}t_{11} &t_{12} & t_{13}\vspace{3pt} \cr
t_{21}& t_{22} &t_{23} \vspace{3pt}  \cr
t_{31} & t_{32} & t_{33}\end{matrix}\right]
\left[\begin{matrix}x \cr y \cr z\end{matrix}\right]=\left[\begin{matrix}\widetilde{x} \cr \widetilde{y} \cr \widetilde{z} \end{matrix}\right].
\end{equation}
Since $(x,y,z)=\lambda (x,y,z)\in \CC\PP^2$, the above invertible $3\times 3$ matrix representing $P$ is determined up to a nonzero scalar. Another name for a projective transformation is a \textit{change of coordinates}. We will show in the next paragraph how projective transformations yield classification of conics.

A nice example of the interplay between curves and linear algebra is the representation of conics with symmetric quadratic forms,
which we will use in the following sections.
We can write each quadratic bivariate polynomial in the homogeneous form
\begin{equation}\label{eq:pkvad2}
p_2(x,y,z)=a_{00}z^2+a_{10}xz+a_{01}yz+a_{20}x^2+a_{11}xy+a_{02}y^2
\end{equation}
as a symmetric quadratic form
\begin{equation}\label{eq:kvformgeneral}
p_2(x,y,z)=\left[\begin{matrix}x & y & z\end{matrix}\right]
\left[\begin{matrix}a_{20} & {1\over 2}a_{11} & {1\over 2}a_{10}\vspace{3pt} \cr
{1\over 2}a_{11} & a_{02} & {1\over 2}a_{01}\vspace{3pt}  \cr
{1\over 2}a_{10} & {1\over 2}a_{01} & a_{00}\end{matrix}\right]
\left[\begin{matrix}x \cr y \cr z\end{matrix}\right].
\end{equation}
The polynomial $p_2$ is decomposable if and only if the corresponding $3\times 3$
symmetric matrix in quadratic form \eqref{eq:kvformgeneral} is degenerate. Clearly, when $p_2$ is decomposable, it equals to a product of two linear forms, thus its zero locus is a union of two lines or a double line.
Over $\CC$ there exsists an invertible matrix $P$ such that
$$ P^{T}\cdot\! \left[\begin{matrix}a_{20} & {1\over 2}a_{11} & {1\over 2}a_{10}\vspace{3pt} \cr
{1\over 2}a_{11} & a_{02} & {1\over 2}a_{01}\vspace{3pt}  \cr
{1\over 2}a_{10} & {1\over 2}a_{01} & a_{00}\end{matrix}\right]\! \cdot P \mbox{ is one of the following }\
\Id,\
 \left[\begin{matrix}1 & 0 &0 \vspace{3pt}  \cr
0 &1 &0\vspace{3pt}  \cr
0& 0 & 0 \end{matrix}\right]
\mbox{ or }
 \left[\begin{matrix}1 & 0 &0 \vspace{3pt}  \cr
0 &0 &0\vspace{3pt}  \cr
0& 0 & 0 \end{matrix}\right].
  $$
  Then $P^{-1}$ defines a projective transformation and $p_2$ becomes in the changed coordinates either
 $\widetilde{x}^2+\widetilde{y}^2+\widetilde{z}^2,\ \widetilde{x}^2+\widetilde{y}^2=(\widetilde{x}+i\widetilde{y})(\widetilde{x}-i\widetilde{y})$ or $\widetilde{x}^2$.  Indeed,
$$p_2(x,y,z)=\left[\begin{matrix}x & y & z\end{matrix}\right]\!
\left[\begin{matrix}a_{20} & {1\over 2}a_{11} & {1\over 2}a_{10}\vspace{3pt} \cr
{1\over 2}a_{11} & a_{02} & {1\over 2}a_{01}\vspace{3pt}  \cr
{1\over 2}a_{10} & {1\over 2}a_{01} & a_{00}\end{matrix}\right]\!
\left[\begin{matrix}x \cr y \cr z\end{matrix}\right]=\left[\begin{matrix} \widetilde{x} & \widetilde{y} & \widetilde{z} \end{matrix}\right]P^T\! \cdot\!
\left[\begin{matrix}a_{20} & {1\over 2}a_{11} & {1\over 2}a_{10}\vspace{3pt} \cr
{1\over 2}a_{11} & a_{02} & {1\over 2}a_{01}\vspace{3pt}  \cr
{1\over 2}a_{10} & {1\over 2}a_{01} & a_{00}\end{matrix}\right]\! \cdot P\!
\left[\begin{matrix}\widetilde{x} \cr \widetilde{y} \cr \widetilde{z} \end{matrix}\right].$$
Note that we could have chosen such $P$ that reduces the symmetric  quadratic form to either
$\widetilde{y}^2-\widetilde{x}\widetilde{z},\ \widetilde{x}\widetilde{y}$ or $\widetilde{x}^2$ defined with matrices
$$\left[\begin{matrix}0 & 0 &-{1\over 2} \vspace{3pt}  \cr
0 &1 &0 \vspace{3pt}  \cr
-{1\over 2} & 0 & 0 \end{matrix}\right] ,\ \
 \left[\begin{matrix}0 & 1 &0 \vspace{3pt}  \cr
1 &0 &0\vspace{3pt}  \cr
0& 0 & 0 \end{matrix}\right] \
\mbox{ or }\
 \left[\begin{matrix}1 & 0 &0 \vspace{3pt}  \cr
0 &0 &0\vspace{3pt}  \cr
0& 0 & 0 \end{matrix}\right]\ \mbox{ respectively}.
$$
Let us recall how the above can be applied to the geometry of  pencils of conics~\cite{reid},~\cite{shafarevich} .
Pick $p_2(x,y,z)$ as above and another conic $q_2(x,y,z)=b_{00}z^2+b_{10}xz+b_{01}yz+b_{20}x^2+b_{11}xy+b_{02}y^2$.  When $p_2$ and $q_2$ have no common components, by B\'ezout's theorem the intersection $\{ (x,y,z)\,:\, p_2(x,y,z)=q_2(x,y,z)=0 \}$ consists of 4 points, counted with multiplicities. Generically, these 4 points are distinct, in which case no three are collinear.
For $(s,t)\in \PP^1$ consider the pencil of conics defined by
\begin{equation}\label{eq:pencil}
s\, p_2(x,y,z)+t\, q_2(x,y,z)=0.\end{equation}
Assume that there are 4 distinct points $T_i=(\lambda_i,\mu_i,\nu_i)$ for $i=1,2,3,4$ in the intersection of $p_2$ and $q_2$.
  Every conic of the pencil~\eqref{eq:pencil}  passes through these four points $T_1,T_2,T_3$, and $T_4$ as shown left  in Figure~\ref{BorPict56}.
It follows that
\begin{equation}\label{eq:demugeneral}
\det\left(
s \left[\begin{matrix}
2a_{20} & a_{11} & a_{10}\cr
a_{11} & 2a_{02} & a_{01}\cr
a_{10} & a_{01} & 2a_{00}\end{matrix}\right]
+t
\left[\begin{matrix}
2b_{00} & b_{11} & b_{10}\cr
b_{11} & 2b_{02} & b_{01}\cr
b_{10} & b_{01} & 2b_{00}\end{matrix}\right]
\right)
\end{equation}
is a homogeneous cubic polynomial in  $s,t$ that equals 0
for exactly three choices of $(s,t)$. For these $(s,t)$ the conic~\eqref{eq:pencil} degenerates into the pairs of lines
$\mathcal{L}(T_1,T_2)\cup \mathcal{L}(T_3,T_4),\
\mathcal{L}(T_1,T_3)\cup \mathcal{L}(T_2,T_4)$ or
$\mathcal{L}(T_1,T_4)\cup \mathcal{L}(T_2,T_3)$.
Here $\mathcal{L}(T_i,T_j)=\{\ell_{ij}(x,y,z)=0\}$
denotes a line through $T_i$ and $T_j$, where
$$\ell_{ij}(x,y,z)=(\mu_i \nu_j-\mu_j \nu_i )x+( \lambda_j \nu_i - \lambda_i \nu_j)y+( \lambda_i \mu_j -\lambda_j \mu_i )z.$$

Conversly, given a $4-$tuple of points $T_1,T_2,T_3,T_4$ such that no three are collinear,
two of the above pairs of lines define
the whole pencil of conics through these 4 points.
For example,
$$
s' \ell_{12}(x,y,z)\ell_{34}(x,y,z)+t'\ell_{13}(x,y,z)\ell_{24}(x,y,z)
$$
 defines the same pencil as equation~\eqref{eq:pencil}.

\begin{figure}
\begin{center}
\includegraphics[width=7cm]{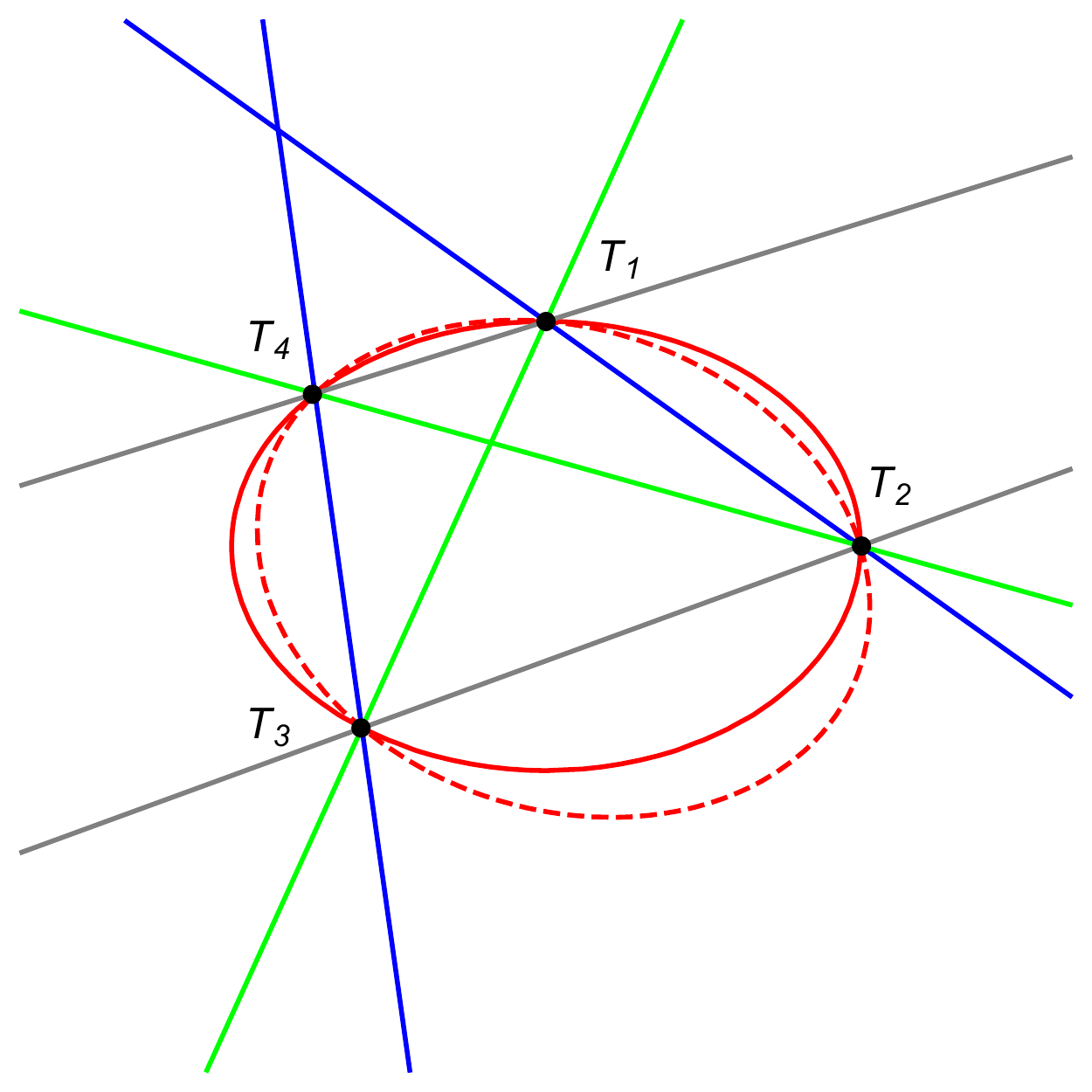}
\includegraphics[width=7cm]{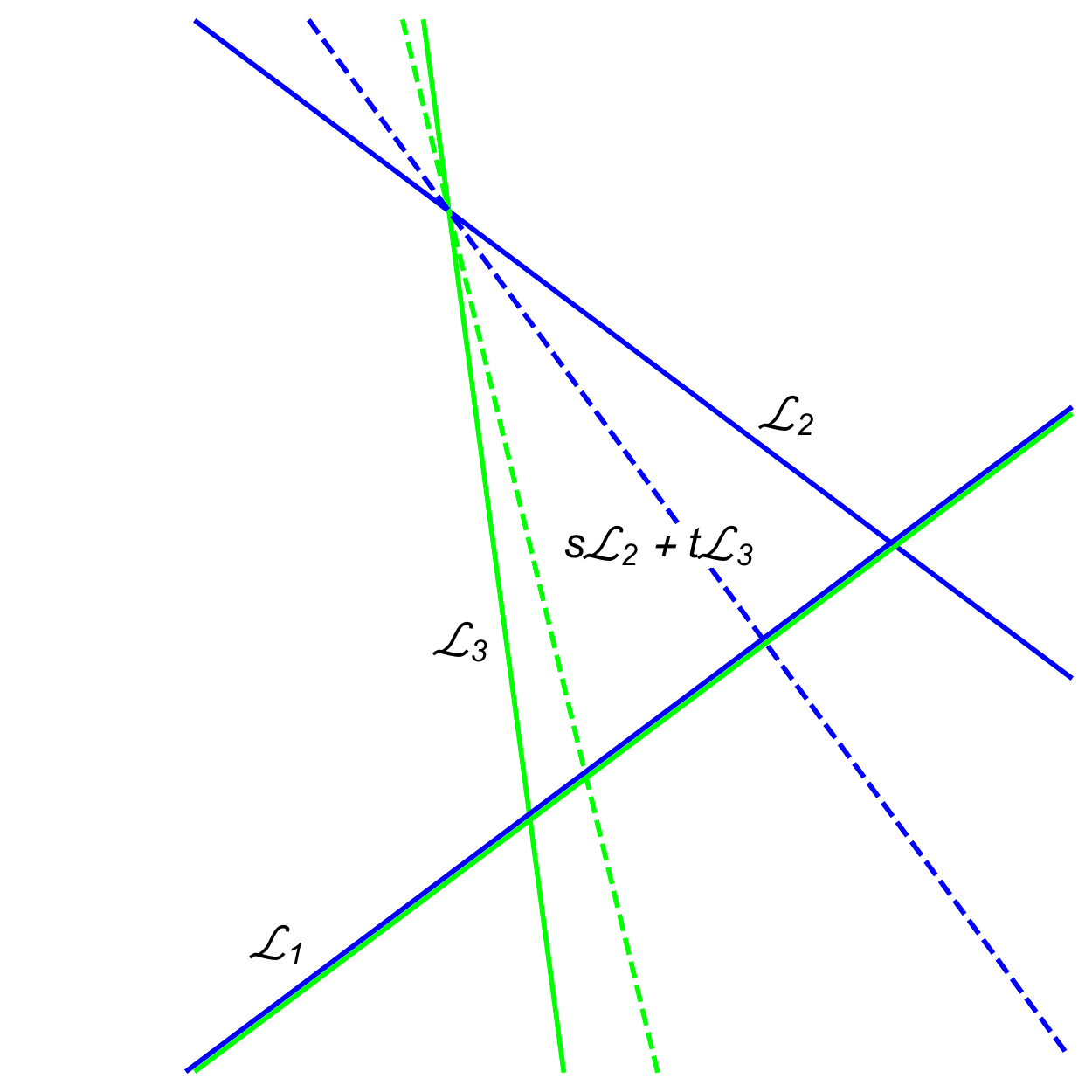}
\end{center}
\caption{A pencil of quadrics through $T_1,T_2,T_3,T_4$. A degenerate pencil of quadrics containing line $\mathcal{L}_1$. \label{BorPict56}}
\end{figure}

In Section~\ref{sec:quartic} we will come across two extremal  pencils of conics, a degenerate pencil in which all conics degenerate and a pencil with only one degenerate member.
The pencil of conics $s\, p_2(x,y,z)+t\, q_2(x,y,z)=0$ is \textit{degenerate} if the determinant~\eqref{eq:demugeneral}  is identically zero. This means that all the quadrics in the pencil are degenerate. On the other hand, for a pencil with only one degenerate conic the determinant of the corresponding quadratic form~\eqref{eq:demugeneral} equals $(\alpha\, s+\beta\, t)^3\neq 0$ for some $\alpha, \beta\in \CC$.
\medskip

\begin{lemma}\label{lem:degpencil}
Consider the
pencil of conics $s\, p_2(x,y,z)+t\, q_2(x,y,z)=0$ defined by  degenerate conics
\begin{align*}
p_2(x,y,z)&=(\alpha_lx+\beta_1y+\gamma_1z)(\alpha_2x+\beta_2y+\gamma_2z),\\
q_2(x,y,z)&=(\alpha_3x+\beta_3y+\gamma_3z)(\alpha_4x+\beta_4y+\gamma_4z).
\end{align*}
The pencil
is degenerate if and only if
 $p_2$ and $q_2$ either have a common factor or all
 the lines $\alpha_i x+\beta_i y + \gamma_i z =0$ for $i=1,2,3,4$ intersect in one point.
\end{lemma}

\begin{proof} We can always apply a change of variables
so that $q_2$ equals either $xy$ or $x^2$.
Firstly, when $q_2(x,y,z)=xy$,
\begin{equation}\label{eq:demuxy}
\det\left(
s \left[\begin{matrix}
2a_{20} & a_{11} & a_{10}\cr
a_{11} & 2a_{02} & a_{01}\cr
a_{10} & a_{01} & 2a_{00}\end{matrix}\right]
+t
\left[\begin{matrix}
0&1 & 0\cr
1& 0 &0\cr
0& 0 & 0\end{matrix}\right]
\right)=0
\end{equation}
for all $(s,t)\in\PP^1$ if  and only if the left matrix in~\eqref{eq:demuxy} is one of the following
$$ \left[\begin{matrix}
2a_{20} & a_{11} & 0\cr
a_{11} & 2a_{02} &0\cr
0 & 0 &0\end{matrix}\right],\quad  \left[\begin{matrix}
2a_{20} & a_{11} & a_{10}\cr
a_{11} & 0 & 0\cr
a_{10} & 0 & 0\end{matrix}\right] \quad \mbox{or}\quad  \left[\begin{matrix}
0 & a_{11} & 0\cr
a_{11} & 2a_{02} & a_{01}\cr
0 & a_{01} & 0\end{matrix}\right].$$
The corresponding quadratic form $p_2$ is then respectively
$$a_{20} x^2+a_{11}xy+a_{02}y^2,\ \  x(a_{20} x+a_{11}y+a_{10}z)\ \mbox{ or }\
y(a_{11} x+a_{02}y+a_{01}z).$$
Secondly, when  $q_2(x,y,z)=x^2$,
\begin{equation}\label{eq:demuxx}
\det\left(
s \left[\begin{matrix}
2a_{20} & a_{11} & a_{10}\cr
a_{11} & 2a_{02} & a_{01}\cr
a_{10} & a_{01} & 2a_{00}\end{matrix}\right]
+t
\left[\begin{matrix}
2&0 & 0\cr
0& 0 &0\cr
0& 0 & 0\end{matrix}\right]
\right)=0
\end{equation}
for all $(s,t)\in\PP^1$ if  and only if the left matrix in~\eqref{eq:demuxx} is one of the following
$$ \left[\begin{matrix}
2a_{20} & 0 & a_{10}\cr
0 & 0 &0\cr
 a_{10} & 0 & 2 a_{00} \end{matrix}\right],\quad \left[\begin{matrix}
2a_{20} & a_{11} & a_{10}\cr
a_{11} & 0 & 0\cr
a_{10} & 0 & 0\end{matrix}\right] \quad \mbox{or}\quad
 \left[\begin{matrix}
2a_{20} & a_{11} & a_{11}a_{01}/(2 a_{02})\cr
a_{11} & 2a_{02} & a_{01}\cr
 a_{11} a_{01}/(2 a_{02}) & a_{01} & 2a_{00}\end{matrix}\right],$$
where $a_{02}\neq 0$ and $a_{01}^2 -4 a_{02} a_{00}=0$.
The corresponding $p_2$ is respectively
$$a_{20} x^2+a_{10}x z+a_{00}z^2,\quad  x(a_{20} x+a_{11}y+a_{10}z),$$
or
$$
a_{02} \left( y +  \frac{
 a_{11} + \sqrt{a_{11}^2 - 4 a_{20} a_{02}}}{2 a_{02}}  x +  \frac{a_{01}}{2 a_{02}} z\right) \left(y + \frac{
 a_{11} - \sqrt{a_{11}^2 - 4 a_{20} a_{02}}}{2 a_{02}}   x +  \frac{a_{01}}{2 a_{02}} z\right).
$$
\end{proof}

\begin{lemma} \label{lem:extrpencil} Let
$s\, p_2(x,y,z)+t\, q_2(x,y,z)=0$ be a pencil of conics with
 $$q_2(x,y,z)=(\alpha_3x+\beta_3y+\gamma_3z)(\alpha_4x+\beta_4y+\gamma_4z).$$
 Then $q_2$ is the only degenerate conic in the pencil if and only if one of the lines
 $\alpha_3 x+\beta_3 y + \gamma_3 z =0$ or $\alpha_4 x+\beta_4 y + \gamma_4 z =0$ is tangent to $p_2$ at their intersection point
$$\{\alpha_3x+\beta_3y+\gamma_3z=0\} \cap \{\alpha_4x+\beta_4y+\gamma_4z=0\}$$ as shown on Figure~\ref{BorPict78}.
\end{lemma}
\begin{proof} As in the proof of Lemma~\ref{lem:degpencil},
we first assume that $q_2(x,y,z)=x y$.
When
\begin{equation}\label{eq:demuisoxy}
\det\left(
s \left[\begin{matrix}
2a_{20} & a_{11} & a_{10}\cr
a_{11} & 2a_{02} & a_{01}\cr
a_{10} & a_{01} & 2a_{00}\end{matrix}\right]
+t
\left[\begin{matrix}
0&1 & 0\cr
1& 0 &0\cr
0& 0 & 0\end{matrix}\right]
\right)=\kappa s^3 \mbox{ for some }\kappa\neq 0,
\end{equation}
 the left matrix in~\eqref{eq:demuisoxy} needs to be
 $$ \left[\begin{matrix}
2a_{20} & a_{11} &0\cr
a_{11} & 2a_{02} & a_{01}\cr
0 & a_{01} &0\end{matrix}\right]\quad\mbox{or}\quad
\left[\begin{matrix}
2a_{20} & a_{11} & a_{10}\cr
a_{11} & 2a_{02} & 0\cr
a_{10} & 0 & 0\end{matrix}\right].$$
Then~\eqref{eq:demuisoxy} equals  $-2 a_{20}  a_{01}^2\, s^3$ or $-2  a_{02} a_{10}^2\,  s^3$,
and the corresponding quadratic form $p_2$ is either
$$ a_{20} x^2 + a_{11} x y + a_{02} y^2 + a_{01} y z\ \ \mbox{ or }\ \ a_{20} x^2 + a_{11} x y + a_{02} y^2 + a_{10} x z. $$
 In both cases $p_2$ is indecomposable and its zero locus is an irreducible conic with the tangent at $(0,0,1)$  being $\{y=0\}$ or $\{x=0\}$ respectively.

Next we consider $q_2(x,y,z)=x^2$. Then
\begin{equation}\label{eq:demuisoxx}
\det\left(
s \left[\begin{matrix}
2a_{20} & a_{11} & a_{10}\cr
a_{11} & 2a_{02} & a_{01}\cr
a_{10} & a_{01} & 2a_{00}\end{matrix}\right]
+t
\left[\begin{matrix}
1&0 & 0\cr
0& 0 &0\cr
0& 0 & 0\end{matrix}\right]
\right)=\kappa s^3 \mbox{ for some }\kappa\neq 0
\end{equation}
if and only if $a_{01}^2-4a_{02}a_{00}=0$. The left matrix in~\eqref{eq:demuisoxx} then equals either
$$ \left[\begin{matrix}
2a_{20} & a_{11} & a_{10}\cr
a_{11} & 0 &0 \cr
a_{10} & 0 & 2a_{00}\end{matrix}\right]\quad\mbox{or}\quad
\left[\begin{matrix}
2a_{20} & a_{11} & a_{10}\cr
a_{11} & 2a_{02} & a_{01}\cr
a_{10} & a_{01} & a_{01}^2/(2a_{02})\end{matrix}\right]$$
and the determinant~\eqref{eq:demuisoxx} is either
 $-2 a_{00}  a_{11}^2\, s^3$ or $\frac{( a_{11} a_{01}- 2 a_{10} a_{02})^2}{4  a_{02}}\, s^3$.
 In the first case  the corresponding irreducible conic $p_2$ has equation
 $ a_{20} x^2 +   a_{11} x y +   a_{10} x z +   a_{00} z^2$
 with the tangent $\{x=0\}$ at $(0,1,0)$, and in the second case $p_2$ has equation
 $  a_{20} x^2 +   a_{11} x y +   a_{02} y^2 +   a_{10} x z +   a_{01} y z +\frac{
  a_{01}^2 z^2}{4  a_{02}}$ with the  tangent $\{x=0\}$ at $(0,a_{01},-2 a_{02})$.
\begin{figure}
\begin{center}
\includegraphics[width=7cm]{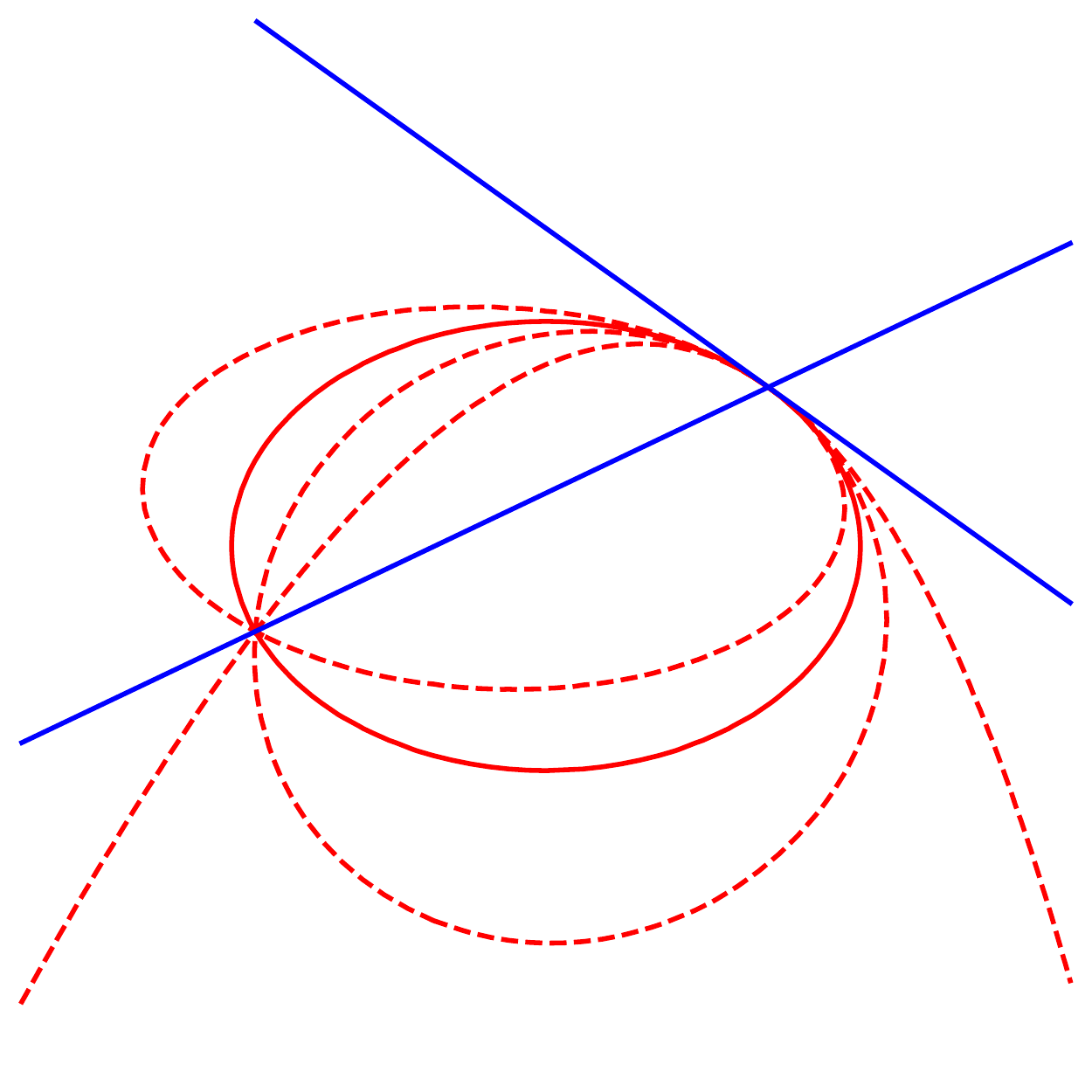}
\includegraphics[width=7cm]{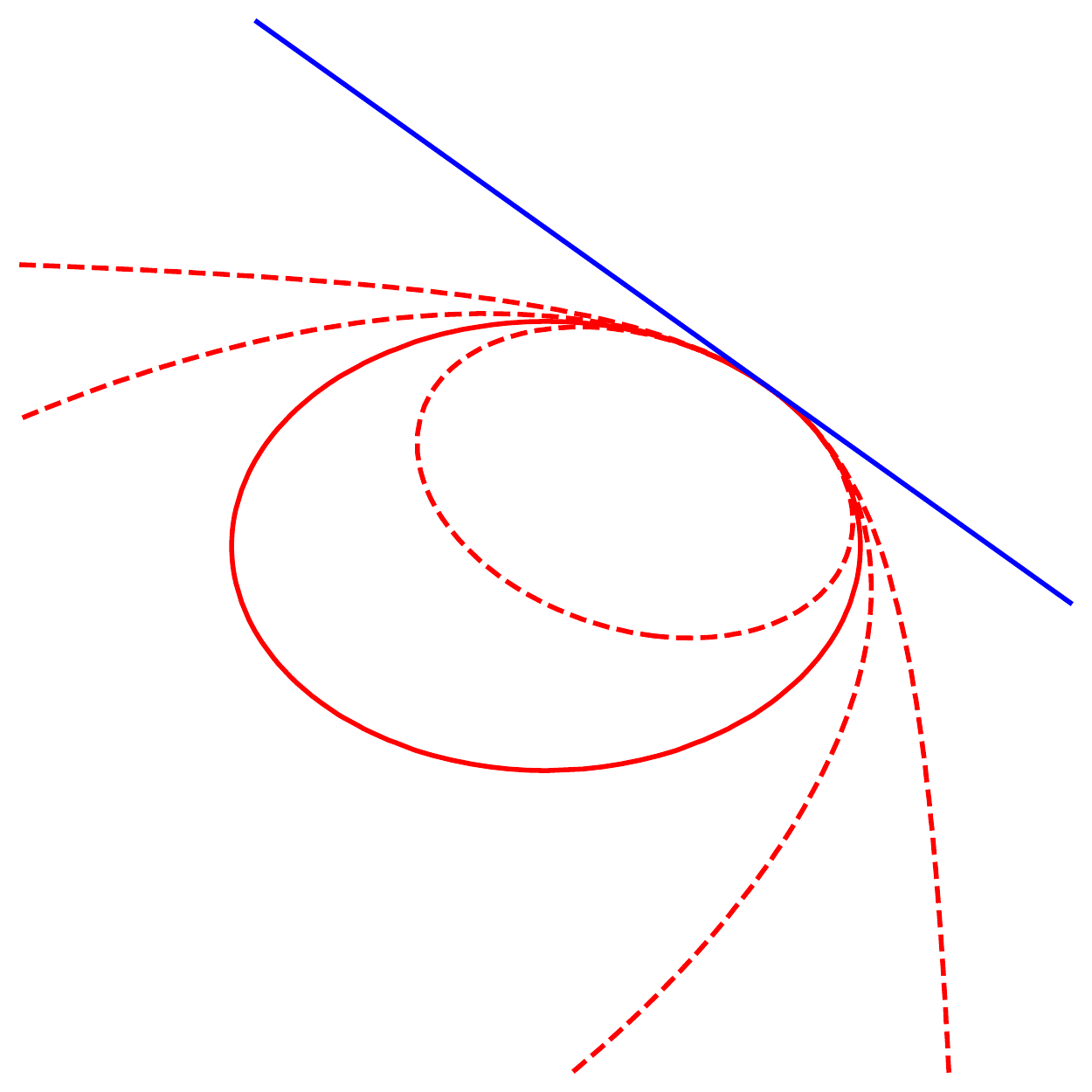}
\end{center}
\caption{Pencils with only one degenerate conic:
$q_2=(\alpha_3x+\beta_3y+\gamma_3z)(\alpha_4x+\beta_4y+\gamma_4z)$ and
$q_2=(\alpha_3x+\beta_3y+\gamma_3z)^2$.   \label{BorPict78}}
\end{figure}
\end{proof}

Suppose that we have $(s,t)$ for which the determinant~\eqref{eq:demugeneral}  equals $0$.
This means that the corresponding quadratic form is decomposable  and we would like to extract its linear factors. The following is an algorithm for this purpose.

\noindent\rule{\textwidth}{1pt}
\noindent {\bf Algorithm 1.} For a quadratic bivariate polynomial $p_2$ in the homogeneous form
\eqref{eq:pkvad2} such that the
corresponding $3\times 3$
symmetric matrix in the quadratic form \eqref{eq:kvformgeneral} is singular,
the algorithm returns linear homogeneous polynomials
$\ell_i(x,y,z)=r_i x + s_i y + t_i z$ for $i=1,2$ such that
$
p_2(x,y,z)=\ell_1(x,y,z)\ell_2(x,y,z).
$
\\[-0.5em]
\noindent\rule{\textwidth}{0.5pt}
\begin{enumerate}
\item[(1)] If $(a_{00},a_{20},a_{02})\ne (0,0,0)$ then
\begin{enumerate}
\item[a)] If $a_{20}\ne 0$, apply a permutation of variables $x,y,z$ such that $a_{20}$ becomes nonzero.
\item[b)] Compute the roots $\alpha_1,\alpha_2$ of $p_2(\alpha,1,0)=a_{02}+a_{11}\alpha + a_{20}\alpha^2=0$
and the roots $\beta_1,\beta_2$  of
$p_2(\beta,0,1)=a_{00}+a_{10}\beta + a_{20}\beta^2=0.$
\item[b)] If $|a_{01}+a_{20}(\alpha_1\beta_1+ \alpha_2\beta_2)|<
|a_{01}+a_{20}(\alpha_1\beta_2+ \alpha_2\beta_1)|,$
exchange $\beta_2$ and $\beta_1$.
\item[d)] Set $\ell_1(x,y,z)=a_{20}(x-\alpha_1 y-\beta_1 z)$ and
$\ell_2(x,y,z)=x-\alpha_2 y-\beta_2 z$.
\end{enumerate}
\item[(2)]  Else
\begin{enumerate}
\item[a)] If $a_{10}\ne 0$, apply a permutation of variables $x,y,z$ such that $a_{10}$ becomes zero.
\item[b)] Set $\ell_1(x,y,z)=y$ and $\ell_2(x,y,z)=a_{11}x+a_{01}z$.
\end{enumerate}
\item[(3)] If a permutation was applied in (1a) or (2b), permute back the variables in $\ell_1$ and $\ell_2$.
\end{enumerate}\vspace{-0.5em}
\noindent\rule{\textwidth}{0.5pt}

Some comments:
\begin{itemize}
\item If $a_{00}=a_{20}=a_{02}=0$, then polynomial $p_2$ has the form
$$ a_{10}xz+a_{01}yz+a_{11}xy= {1\over 2}
\left[\begin{matrix}x & y & z\end{matrix}\right]
\left[\begin{matrix}0 &a_{11} &a_{10}\vspace{3pt} \cr
a_{11} & 0& a_{01}\vspace{3pt}  \cr
a_{10} & a_{01} & 0 \end{matrix}\right]
\left[\begin{matrix}x \cr y \cr z\end{matrix}\right].
$$
The matrix of the above symmetric form is clearly singular if and only if
$a_{11}a_{01}a_{10}=0$. Therefore, we can always find a permutation of variables
in Step (2a) that makes $a_{10}=0$.
\item The obtained decomposition is not unique as we can always
respectively replace $(r_1,s_1,t_1)$ and $(r_2,s_2,t_2)$
by
$(\lambda r_1,\lambda s_1,\lambda t_1)$ and
$(r_2/\lambda,s_2/\lambda,t_2/\lambda)$
for a nonzero $\lambda$.
\item Polynomial $p_2$ is decomposable if and only if the rank of the symmetric matrix
   in \eqref{eq:kvformgeneral}  is $1$ or $2$.
   In addition, the rank is 1 exactly when $p_2$ is a square of a linear homogeneous
   polynomial. In this case we can simply take
    $\ell_1(x,y,z)=\ell_2(x,y,z)=\root \of{a_{20}}\, x + \root \of{a_{02}}\,
y + \root \of{a_{00}}\, z$.
\end{itemize}
\bigskip

\section{Reduction}\label{sec:ncurve}

Let $p_n$ be a bivariate polynomial of degree $n$ in the homogeneous
form
\begin{equation}\label{eq:pn}
p_n(x,y,z)=a_{00}z^n+a_{10}xz^{n-1}+a_{01}yz^{n-1}+\cdots+a_{n0}x^n+\cdots+a_{0n}y^n,
\end{equation}
which means that at least one of the coefficients $a_{n0},a_{n-1,1},\ldots,a_{0n}$ is nonzero.
Its zero locus
\begin{equation}\label{eq:locus}
\mathcal{C}=\left\{ (x,y,z)\in\CC\PP^2\, : \, p_n(x,y,z)=0 \right\}
\end{equation}
 defines a projective plane curve of degree $n$.

We can assume that $a_{n0}\ne 0$. The geometric meaning of  $a_{n0}\ne 0$ is that  $(1,0,0)\notin \mathcal{C}$. If
$a_{n0}=0$, we apply a change of variables
\begin{equation}\label{eq:substit}
\left[\begin{matrix}
x \cr
y  \cr
 z
\end{matrix}\right]=\left[\begin{matrix}
 c & s &0 \cr
 -s  & c & 0 \cr
 0 &0 &1
\end{matrix}\right]\, \left[\begin{matrix}
\widetilde{x} \cr
 \widetilde{y}  \cr
 \widetilde{z}
\end{matrix}\right],
\end{equation}
 such that $c^2+s^2=1$ and that
the coefficient at $\widetilde x^n$ of the substituted
polynomial $\widetilde  p_n(\widetilde  x, \widetilde  y,\widetilde  z)$
is nonzero. Indeed, the coefficient at $\widetilde x^n$ equals
to $p_n(c,s,0)$ and we can choose $c=\cos\varphi$ and $s=\sin\varphi$ such that
$p_n(\cos\varphi,\sin\varphi,0)\ne 0$. The substitution
\eqref{eq:substit} corresponds to a rotation of the coordinates $x,y$ around $z$. Such
transformations are also used in \cite{Jonsson} due to their numerical
stability.
After we construct a determinantal representation for the substituted
polynomial in $\widetilde x,\widetilde y$ and $\widetilde z$,
we perform the substitution back to $x,y$ and $z$.
\medskip

\begin{lemma}\label{lem:reduk1}
Let $p_n$ be a bivariate polynomial of degree $n$ in the homogeneous
form \eqref{eq:pn} such that $a_{n0}\ne 0$. If $\alpha_1,\alpha_2,\ldots,\alpha_n$ are the roots of
$$p_n(\alpha,1,0)=a_{n 0} \alpha^n+a_{n-1, 1}\alpha^{n-1}  +\cdots+a_{1, n-1}\alpha+a_{0 n}$$
and
$\beta_1,\beta_2,\ldots,\beta_n$ are the roots of
$$p_n(\beta,0,1)=a_{00}+a_{10}\beta + a_{20}\beta^2+\cdots+a_{n0}\beta^n,$$
then
\begin{equation}\label{eq:qn}
p_n(x,y,z)-a_{n0}\prod_{j=1}^n (x-\alpha_j y -\beta_j z)=y\,z\,q_{n-2}(x,y,z),
\end{equation}
where $q_{n-2}$ is a homogeneous polynomial of degree $n-2$.
\end{lemma}

\begin{proof}
Pick the line $\mathcal{L}_z=\left\{z=0  \right\}$ in $\CC\PP^2$
and consider the intersection $\mathcal{C}\cap \mathcal{L}_z$.
Since $a_{n0}\ne 0$, the $n$ points in the intersection (counted with multiplicities) are
$\{(\alpha_j,1,0)\}_{j=1,\ldots,n}$.
Analogously, the set of points $\{(\beta_j,0,1)\}_{j=1,\ldots,n}$ equals the intersection of
$\mathcal{C}$ with the line $\mathcal{L}_y=\left\{ y=0  \right\}.$

The reduction \eqref{eq:qn} follows from the construction of
$\alpha_1,\alpha_2,\ldots,\alpha_n$ and $\beta_1,\beta_2, \ldots,\beta_n$.
Indeed, the zero locus of $\prod_{j=1}^n(x-\alpha_j y -\beta_j z)$ is a union of $n$ lines (dashed grey on Figure~\ref{BorPict1})
$$\mathcal{L}_j=\left\{ x-\alpha_j y -\beta_j z=0 \right\} \mbox{ through } (\alpha_j,1,0) \mbox{ and } (\beta_j,0,1).$$
Then the set of zeros  $$\left\{ (x,y,z)\in \CC\PP^2 \,:\,p_n(x,y,z)-a_{n0}\prod_{j=1}^n (x-\alpha_j y -\beta_j z)=0 \right\}$$ contains
$n+1$ points $\{ (\alpha_j,1,0)\}_{j=1,\ldots,n}$ and $(1,0,0)$ on the line $\mathcal{L}_z$; and
$n+1$ points $\{ (\beta_j,0,1)\}_{j=1,\ldots,n}$ and $(1,0,0)$ on the line $\mathcal{L}_y$. Therefore, by B\'ezout's theorem,  it contains both lines
$\mathcal{L}_z\cup \mathcal{L}_y$ and it can
be presented as the zero locus of $y z q_{n-2}(x,y,z)$, where
$q_{n-2}$ is a polynomial of degree $n-2$.
\end{proof}

\begin{figure}
\begin{center}
\includegraphics[width=8cm]{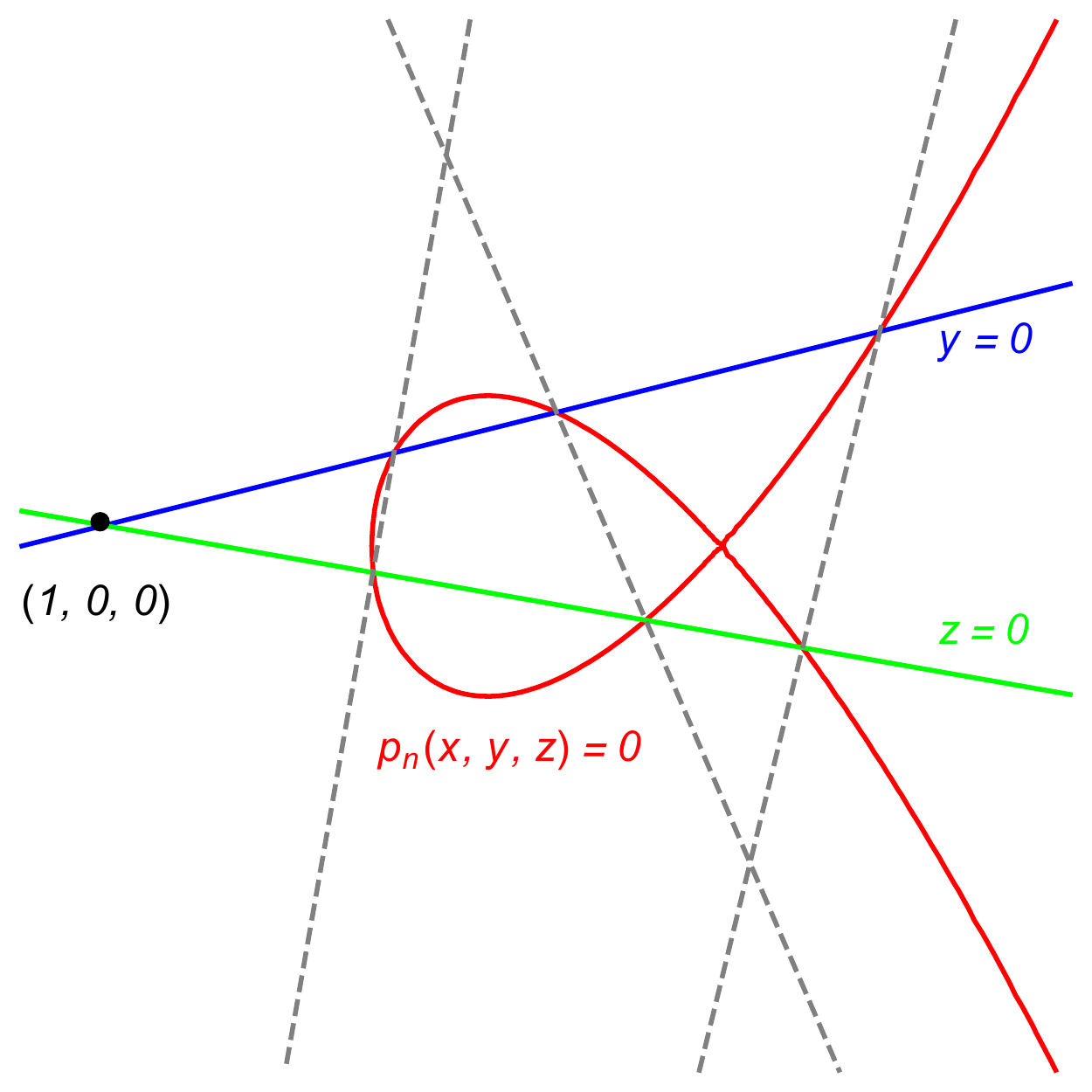}
\end{center}
\caption{The lines $y=0$ and $z=0$ intersecting in $(1,0,0)\notin \mathcal{C}$.\label{BorPict1}}
\end{figure}

The reduction \eqref{eq:qn} is a vital key for our constructions of determinantal
representations with $n\times n$ matrices. By using the reduction and a careful
placement of linear polynomials in an $n\times n$ matrix pencil, we  show
in the following sections how to construct a representation of the minimal
possible size for all polynomials up to degree $n=5$.

\begin{remark}\label{rem:notline}{\rm
In the sequel we always assume that $p_n(x,y,z)$ does not define a line, in other words
$p_n(x,y,z)\ne \left(\root n\of {a_{n0}}\, x + \root n\of {a_{0n}}\,
y + \root n\of {a_{00}}\, z\right)^n$. Such form, which is easy to detect, gives
a straightforward determinantal representation with diagonal matrices.

Once the option that ${\cal C}$ is a line is excluded, we know
that a generic line intersects ${\cal C}$ in at least two distinct
points. If necessary, we make a preliminary change of variables so that each of
$\mathcal{L}_y$ and $\mathcal{L}_z$ intersects $\mathcal{C}$ in at least two distinct points. In other words, none of the points $(\alpha_i,1,0),(\beta_j,0,1)$ have order $n$.
}\end{remark}\\

It is natural to ask whether $a_{n 0}\neq 0$ is a necessary condition to obtain
the reduction~(\ref{eq:qn}).
Note that $a_{n0}= 0$ if and only if $(1,0,0)\in \mathcal{C}$.
 When $(1,0,0)\in \mathcal{C}$ is a smooth point, we can compute the tangent
\begin{align}\label{eq:qtang}
T \mathcal{C}_{(1,0,0)}&=\left\{ \frac{\partial p_n}{\partial x}(1,0,0)\, x+ \frac{\partial p_n}{\partial y}(1,0,0)\, y+ \frac{\partial p_n}{\partial z}(1,0,0)\, z=0    \right\}\\
&=\left\{ a_{n-1, 1}y+a_{n-1, 0}z=0    \right\}\nonumber.
\end{align}
Note that $ a_{n-1, 1}\, a_{n-1, 0}\neq 0$ if and only if the tangent is neither of  the lines $\mathcal{L}_y,\ \mathcal{L}_z$.
In this case there exists a similar reduction of the polynomial $p_n$ as in Lemma~\ref{lem:reduk1}.
\medskip

\begin{lemma}\label{lem:tang}
Let $p_n$ be a bivariate polynomial of degree $n$ in the homogeneous
form \eqref{eq:pn} such that $(1,0,0)$ is a smooth point of its zero locus \eqref{eq:locus}.
If $a_{n-1, 1}\, a_{n-1, 0}\neq 0$, there exists a polynomial $q_{n-2}$ of degree $n-2$ such that $p_n$ reduces to
\begin{equation}\label{eq:qntan}
p_n(x,y,z)-( a_{n-1, 1}y+a_{n-1, 0}z) \prod_{j=1}^{n-1} (x-\alpha_j y -\beta_j z)=y\,z\,q_{n-2}(x,y,z),
\end{equation}
where $\alpha_1,\ldots,\alpha_{n-1}$ are the roots of $p(\alpha,1,0)=a_{n-1, 1}\alpha^{n-1}  +\cdots+a_{1, n-1} \alpha +a_{0\, n}=0$
and $\beta_1,\ldots,\beta_{n-1}$ are the roots of $p(\beta,0,1)=a_{00} +a_{10}\beta+ a_{20}\beta^2+\cdots+a_{n-1, 0}\beta^{n-1}=0$.
\end{lemma}

\begin{proof}
The intersection of $\mathcal{C}=\{p_n=0\}$ with $\mathcal{L}_z=\left\{ z=0  \right\}$ consists of $n$ points.
We obtain them as solutions of
$$p_n(x,y,0)=a_{n-1, 1}x^{n-1}y  +\cdots+a_{1, n-1} x y^{n-1}+a_{0 n} y^n=0.$$
Thus,
$$ \mathcal{C}\cap \mathcal{L}_z=\{(\alpha_1,1,0),\ldots,(\alpha_{n-1},1,0), (1,0,0)\} , $$
Analogously, the $n$ solutions of
$$p_n(x,0,z)=a_{00} z^n+a_{10}x z^{n-1} + a_{20}x^2 z^{n-2}+\cdots+a_{n-1, 0}x^{n-1} z=0$$
yield the $n$ points in the intersection  $\mathcal{C}\cap \mathcal{L}_y$, which means
$$ \mathcal{C}\cap \mathcal{L}_y=\{ (\beta_{1},0,1),\ldots,(\beta_{n-1},0,1),(1,0,0) \}.$$
The union of
$T \mathcal{C}_{(1,0,0)}$ and the lines
$ \mathcal{L}_j=\left\{ x-\alpha_j y -\beta_j z=0 \right\}$ through $ (\alpha_j,1,0)$  and $ (\beta_j,0,1)$ for $j=1,\ldots ,n-1$
can be presented as the set of zeros of the following polynomial
$$( a_{n-1, 1}y+a_{n-1, 0}z) \prod_{j=1}^{n-1} (x-\alpha_j y -\beta_j z).$$
Then the zero locus
$$\left\{p_n(x,y,z)-( a_{n-1, 1}y+a_{n-1, 0}z) \prod_{j=1}^{n-1} (x-\alpha_j y -\beta_j z)=0 \right\}$$ intersects the line $\mathcal{L}_z$ in
at least $n+1$ points counted with multiplicities: $\{ (\alpha_j,1,0)\}_{j=1,\ldots,n-1}$ and $(1,0,0)$ with multiplicity $\geq 2$; and  it intersects the line $\mathcal{L}_y$ in at least $n+1$ points counted with multiplicities:
$\{ (\beta_j,0,1)\}_{j=1,\ldots,n}$ and $(1,0,0)$ with multiplicity $\geq 2$.
Therefore, by B\'ezout's theorem,  it contains both lines
$\mathcal{L}_z\cup \mathcal{L}_y$ and it can thus
be presented as the zero locus of $y z q_{n-2}(x,y,z)$, where
$q_{n-2}$ is a polynomial of degree $n-2$, which gives the reduction \eqref{eq:qntan}.
\end{proof}
\medskip

The following example shows that when $T \mathcal{C}_{(1,0,0)}$ equals either $\mathcal{L}_y$
or  $\mathcal{L}_z$, it is not possible
to reduce $p_n(x,y,z)$ into $y\,z\,q_{n-2}(x,y,z)$ by the subtraction of
a product of $n$ linear forms.
\medskip

\begin{example}{\rm Consider $ \mathcal{C}$, a cubic  defined by the polynomial $p_3(x,y,z)=x^2 y - y^3 + x z^2 + z^3=0$. From $p_3(x,y,z)=z^2 (x + z) + y (x - y) (x + y)=0$ we obtain
\begin{enumerate}
\item[a)] $T \mathcal{C}_{(1,0,0)}=\mathcal{L}_y=\{y=0\}$,
\item[b)]  $\mathcal{C}\cap \mathcal{L}_z=\{(1,1,0),(-1,1,0),(1,0,0)\}$,
\item[c)]  $\mathcal{C}\cap \mathcal{L}_y=\{(-1,0,1), (1,0,0),(1,0,0)\}$.
\end{enumerate}
Assume that it is possible to reduce $p_3$ into the form
\begin{equation}\label{eq:countercub} p_3(x,y,z)- \ell_1(x,y,z)\, \ell_2(x,y,z)\, \ell_3(x,y,z)=y\,z\,q_{1}(x,y,z),\end{equation}
where $q_{1}(x,y,z)$ and $\ell_i(x,y,z)$ are linear polynomials. The reducible cubic defined by  $\ell_1\, \ell_2\, \ell_3=0$  is a union of lines $\{\ell_i=0\}$ for $i=1,2,3$
containing  the four points  $(1,1,0),(-1,1,0),(1,0,0),(-1,0,1)$.
Then at least one of the lines $\ell_i$ needs to contain two of these points, thus $\ell_i$ is either $y=0,\ z=0,\  x+y+z=0$ or $x-y+z=0$. Since $p_3(x,y,z)$ is not divisible by $y$ or $z$, it must hold $\ell_i=x+y+z$ or $\ell_i=x-y+z$.
By setting $y=0$ and $z=0$  in (\ref{eq:countercub}) and using the fact that $\CC[x,y], \CC[x,z]$ are unique factorization domains,  we obtain a contradiction with the possible $\ell_1,\ell_2,\ell_3$.
}\end{example}

\section{Quadratic polynomials}\label{sec:quadratic}

Let $p_2$ be a quadratic bivariate polynomial in the homogeneous form
\begin{equation}\label{eq:pol2}
p_2(x,y,z)=a_{00}z^2+a_{10}xz+a_{01}yz+a_{20}x^2+a_{11}xy+a_{02}y^2,
\end{equation}
such that $a_{20}\ne 0$. If $\alpha_1,\alpha_2$ are the roots of the quadratic
equation
$$p_2(\alpha,1,0)=a_{02}+a_{11}\alpha + a_{20}\alpha^2=0$$
and similarly $\beta_1,\beta_2$ are the roots of
$$p_2(\beta,0,1)=a_{00}+a_{10}\beta + a_{20}\beta^2=0,$$
then clearly
$$p_2(x,y,z)-a_{20}(x-\alpha_1 y -\beta_1 z)(x-\alpha_2 y -\beta_2 z)=q_0 yz$$
for a scalar $q_0$. This gives
the determinantal representation
\begin{equation}\label{eq:det2}
xA+yB+zC=\left[\begin{matrix} a_{20}(x-\alpha_1y-\beta_1z) & -q_0 y\cr
z & x-\alpha_2y-\beta_2z\end{matrix}\right].
\end{equation}

We remark that in the case $\alpha_1\ne \alpha_2$ and $\beta_1\ne \beta_2$ we obtain
a different representation if we exchange the order of $\beta_1$ and $\beta_2$.
This is important when $p_2$ is a decomposable quadratic polynomial, as then we
can choose the order so that $q_0=0$. Numerically it seems reasonable to
select the order that gives the smallest absolute value of $q_0$.


The following lemma shows that for a quadratic bivariate polynomial in the homogeneous form
\eqref{eq:pol2} there always exists such a $2\times 2$ determinantal representation that
one of its elements is fixed to be $x$. We use this
particular representation in Section \ref{sec:quartic} to construct determinantal representations of
quartic polynomials.

\begin{lemma}\label{lm:kvadrat} Let $p_2$
be a quadratic bivariate polynomial
in the homogeneous form \eqref{eq:pol2}.  Then there exist linear polynomials
$\ell_i(x,y,z)=r_i x + s_i y +t_i z$ for $i=1,2,3$ such that
$$p_2(x,y,z) = \ell_1(x,y,z)\ell_2(x,y,z)-x \ell_3(x,y,z).$$
\end{lemma}
\begin{proof}
First, we consider the case $a_{00}\ne 0$.
After exchanging the roles of $x$ and $z$ in Lemma~\ref{lem:reduk1} we can
subtract
$a_{00}(z-\alpha_1 x-\beta_1 y)(z-\alpha_2 x-\beta_2 y)$ from $p_2(x,y,z)$ to get the
residual $q_0 x y$, where $q_0\in\CC$.
Here
$\alpha_1,\alpha_2$ are the roots of the quadratic equation
$p_2(\alpha,0,1)=0$ and $\beta_1,\beta_2$ are the roots of the
quadratic equation $p(0,\beta,1)=0$. This gives
$$p_2(x,y,z)=a_{00}(z-\alpha_1 x-\beta_1 y)(z-\alpha_2 x-\beta_2 y)+q_0 x y.$$

If $a_{00}=0$ and $a_{02}\ne 0$, we use the same approach as above, only that
we exchange the roles of $y$ and $z$. If
$\alpha_1,\alpha_2$ are the roots of
$p_2(\alpha,1,0)=0$ and $\beta_1,\beta_2$ are the roots of $p_2(0,1,\beta)=0$, then we get $q_0\in\CC$ such that
$$p_2(x,y,z)=a_{02}(y-\alpha_1 x-\beta_1 z)(y-\alpha_2 x-\beta_2 z)+q_0 x z.$$

Finally, if $a_{00}=a_{02}=0$, we can take
$$p_2(x,y,z)=y(a_{01}z+a_{11}x) + x(a_{20}x+a_{10}z).$$
\end{proof}

\section{Cubic polynomials}\label{sec:cubic}

Let $p_3$ be a cubic bivariate polynomial in the homogeneous form
\begin{equation}\label{eq:cubpol}
p_3(x,y,z)=a_{00}z^3+a_{10}xz^2+a_{01}yz^2+\cdots+a_{30}x^3+\cdots+a_{03}y^3
\end{equation}
with
$a_{30}\ne 0$.
Let $\alpha_1,\alpha_2,\alpha_3$ be the roots of
$p_3(\alpha,1,0)=0$
and let $\beta_1,\beta_2,\beta_3$ be the roots of
$p_3(\beta,0,1)=0$.
Then, according to Lemma \ref{lem:reduk1}, there exist easily computable coefficients
$b_{00}, b_{10}$, and $b_{01}$  for which
\begin{equation}\label{eq:cubred}
$$p_3(x,y,z)-a_{30}\prod_{j=1}^3(x-\alpha_j y -\beta_j z)=yz(b_{00}z + b_{10}x+b_{01}y).
\end{equation}
This reduction gives the determinantal representation of $p_3$
\begin{equation}\label{eq:rep3b}
xA+yB+zC=\left[\begin{matrix}a_{30}(x-\alpha_1y-\beta_1 z) & 0 & b_{00}z + b_{10}x+b_{01}y\cr
y & x-\alpha_2y-\beta_2 z & 0 \cr
0 &  z & x-\alpha_3y-\beta_3 z\end{matrix}\right].
\end{equation}

For the quadratic polynomial \eqref{eq:pkvad2} we have a nice way to check whether the
polynomial is a product of linear polynomials. Namely, in such case the symmetric
matrix in the corresponding quadratic form \eqref{eq:kvformgeneral} is singular.
There is no such simple tool for the cubic polynomial \eqref{eq:cubpol}, however we can make
use of  the reduction \eqref{eq:cubred}. When $p_3$ is a product of three
linear polynomials, we can order $\alpha_1,\alpha_2,\alpha_3$ and
$\beta_1,\beta_2,\beta_3$ in such a way that $b_{00}=b_{01}=b_{10}=0$ in \eqref{eq:cubred}. In the
generic case, when all roots $\alpha_1,\alpha_2,\alpha_3$ and
$\beta_1,\beta_2,\beta_3$ are simple, there are 6 possible permutations that we need
to check.
\medskip

\begin{example}\label{ex:Weier}{\rm
It is well known that by a projective change of coordinates, every irreducible cubic curve
can be brought into the Weierstrass form (see, e.g., \cite{shafarevich})
$$yz^2=x(x+\theta_1 y)(x+\theta_2y),$$ where $\theta_1,\theta_2\in\CC$.
The corresponding polynomial is $p_3(x,y,z)=x(x+\theta_1 y)(x+\theta_2y)-yz^2$.
It is easy to see that the above procedure yields the determinantal representation
\begin{equation}\label{eq:wei1}
xA+yB+zC=\left[\begin{matrix}x & 0 & -z\cr
y & x-\theta_1y & 0 \cr
0 &  z & x-\theta_2 y\end{matrix}\right].
\end{equation}
}
\end{example}
\medskip

Determinantal representations are not unique.
If $x A+yB+zC$ and $x A'+yB'+zC'$ are $n\times n$ determinantal representations of the same polynomial, then we call determinantal representations \textit{equivalent} if there exist matrices $P,Q\in \mbox{GL}(n,\CC)$ such that
$$P\cdot (x A+yB+zC)\cdot Q=x A'+yB'+zC'.$$
The following example shows that
permutations of $\alpha_i$ and of $\beta_i$ in Lemma~\ref{lem:reduk1} yield different (nonequivalent) determinantal  representations.
\medskip

\begin{example}\label{ex:permute}{\rm
    Consider the Weierstrass cubic in Example~\ref{ex:Weier}. If $\theta_1\ne \theta_2$, then
it is easy to see that determinantal representations
\[ xA'+yB'+zC'=\left[\begin{matrix}x & 0 & -z\cr
y & x-\theta_2y & 0 \cr
0 &  z & x-\theta_1 y\end{matrix}\right]\]
and \eqref{eq:wei1} are not equivalent. Indeed, $A=A'=I$ implies that $Q=P^{-1}$, and it remains to be verified that $P$ such that
 $PB=B'P$ and $PC=CP$ does not exist.
}\end{example}
\medskip

\begin{remark}\label{cubtheta}{\rm
 All determinantal representations of a smooth cubic curve $\mathcal{C}=\{p_3(x,y,z)=0\}$ can be parametrised by the affine points on $\mathcal{C}$. This follows from the famous Cook and Thomas correspondence~\cite{Cook} between line bundles and determinantal representations. Vinnikov~\cite{VinnikovComplete},~\cite{VinnikovCubic}  explicitely described this  correspondence for cubics in the canonical Weierstrass form. It turns out that
a smooth cubic has exactly three symmetric determinantal representations corresponding to the three even theta characteristics  on $\mathcal{C}$. 

The symmetric determinantal representations of a smooth cubic can be explicitely computed by the following algorithm due to Harris~\cite{HarrisGalois}: there exist precisely three solutions $(a,b)\in \CC^2$ such that
$a\, p_3=\mbox{Hes}\,(b\, p_3+\mbox{Hes}\,(p_3))$, where Hes denotes the determinant of the Hessian matrix. An elementary proof of this construction can be found in~\cite{PlaumannHVcurves}.
Moreover, Harris in~\cite[Chapter II.2]{HarrisGalois} describes a symbolic algorithm for finding the nine flexes of a smooth cubic $\mathcal{C}$. Combining this with Vinnikov's  determinantal representations  of Weierstrass cubics, we can parametrize the whole set of determinantal representations  of $\mathcal{C}$ by the affine points of its corresponding Weierstrass form.
}
\end{remark}


\section{Quartic polynomials}\label{sec:quartic}

Let $p_4$ be a quartic bivariate polynomial in the homogeneous
form
$$p_4(x,y,z)=a_{00}z^4+a_{10}xz^3+a_{01}yz^3+\cdots+a_{40}x^4+\cdots+a_{04}y^4,$$
where as before we  assume that $a_{40}\ne 0$.
Denote by $\alpha_1,\alpha_2,\alpha_3,\alpha_4$ the roots of
$p_4(\alpha,1,0)=0$
and by $\beta_1,\beta_2,\beta_3,\beta_4$ the roots of
$p_4(\beta,0,1)=0$.
The ansatz for a determinantal
representation of $p_4$ is
\begin{equation}\label{eq:rep4b}
\left[\begin{matrix}a_{40}(x-\alpha_1y-\beta_1 z) & -y & 0 & 0\cr
0 & x-\alpha_2y-\beta_2 z & r_1 x + s_1 y + t_1 z & r_3 x + s_3 y + t_3 z\cr
0 & 0 & x-\alpha_3y-\beta_3 z & r_2 x + s_2 y + t_2 z \cr
z & 0 & 0 & x-\alpha_4y-\beta_4 z\end{matrix}\right],$$
\end{equation}
 whose determinant is
\begin{align}
a_{40}\prod_{j=1}^4(x-\alpha_j y-\beta_j z)+y z\det
&\left[\begin{matrix}
 r_1 x + s_1 y + t_1 z & r_3 x + s_3 y + t_3 z \cr
  x-\alpha_3y-\beta_3 z & r_2 x + s_2 y + t_2 z
\end{matrix}\right]\label{eq:det4}\\
=a_{40}\prod_{j=1}^4(x-\alpha_j y-\beta_j z)&+yz(r_1 x + s_1 y + t_1)(r_2 x + s_2 y + t_2)\nonumber \\
&-y z (x-\alpha_3 y-\beta_3 z)(r_3 x + s_3 y + t_3).\nonumber
\end{align}
The idea behind the ansatz is the following. From the construction of
$\alpha_1,\alpha_2,\alpha_3,\alpha_4$ and $\beta_1,\beta_2,\beta_3,\beta_4$  and the reduction~\eqref{eq:qn} in  Lemma \ref{lem:reduk1}  it follows that
\begin{equation}\label{eq:ost4}
p_4(x,y,z)-a_{40}\prod_{j=1}^4(x-\alpha_j y -\beta_j z)=yzq_2(x,y,z),
\end{equation}
where $q_2$ is a polynomial of degree $2$. If we are able to find
linear homogeneous polynomials $r_i x + s_i y + t_i z$ for $i=1,2,3$
such that
\begin{equation}\label{eq:2t2minor}
q_2=\det
\left[\begin{matrix}
 r_1 x + s_1 y + t_1 z & r_3 x + s_3 y + t_3 z \cr
  x-\alpha_3y-\beta_3 z & r_2 x + s_2 y + t_2 z
\end{matrix}\right],
\end{equation}
  then we have a determinantal representation of $p_4$.

It turns out that this is always possible due to Lemma \ref{lm:kvadrat}. Using a substitution of variables
$x=\widetilde x + \alpha_3 \widetilde y+\beta_3\widetilde z$,
$y=\widetilde y$, and $z=\widetilde z$ we
change $q_2(x,y,z)$ into $\widetilde q_2(\widetilde x,\widetilde y,\widetilde z)$.
Now we apply Lemma \ref{lm:kvadrat} to obtain
linear homogeneous polynomials
$\widetilde\ell_i(\widetilde x,\widetilde y,\widetilde z)=
\widetilde r_i \widetilde x + \widetilde s_i \widetilde y +\widetilde t_i \widetilde z$
for $i=1,2,3$ such that
$$\widetilde q_2(\widetilde x,\widetilde y,\widetilde z) =
\widetilde\ell_1(\widetilde x,\widetilde y,\widetilde z)\widetilde \ell_2(\widetilde x,\widetilde y,\widetilde z)-\widetilde x \widetilde \ell_3(\widetilde x,\widetilde y,\widetilde z).$$
When we change back the variables, we get
$\ell_i(x,y,z)=r_i x + s_i y + t_i z$ from $\widetilde\ell_i$ for $i=1,2,3$.
The determinant of \eqref{eq:rep4b} is $p_4(x,y,z)$, thus we have constructed a determinantal representation of $p_4$.
\bigskip

Although we have already shown above how to construct a determinantal representation for a
quartic polynomial, let us consider another possible approach. We use the same ansatz,
but take a different path to construct $r_i x + s_i y + t_i z$ for $i=1,2,3$.

The main idea in this alternative approach is to select a nonzero
linear polynomial $\rho x + \sigma y + \tau z$ and then
perturb $q_2$ with $\mu (x-\alpha_3 y-\beta_3 z)(\rho x + \sigma y + \tau z)$ to make
it decomposable. This means that
we choose such parameter $\mu$ that the
difference $$m_2(x,y,z)=q_2(x,y,z)-\mu (x-\alpha_3 y-\beta_3 z)(\rho x + \sigma y + \tau z)$$ is
a product of two linear factors
\begin{equation}\label{eq:razc2b}
m_2(x,y,z) = (r_1 x + s_1 y + t_1 z)(r_2 x + s_2 y + t_2 z).
\end{equation}
Once we have $\mu$ and compute $m_2$, we can use Algorithm 1 to obtain the factors in \eqref{eq:razc2b}.

 %
%

We compute $\mu$ by applying the pencils of conics discussed in Section~\ref{sec:planecurves}.
Consider the pencil
$$s\, q_2(x,y,z)+t\,  (x-\alpha_3 y-\beta_3 z)(\rho x + \sigma y + \tau z).$$
We showed that there exist three (possibly multiple) choices of $(s,t)\in \CC\PP^1$ for which the pencil degenerates;
clearly $(0,1)$ is one of them. For a generic $\rho x + \sigma y + \tau z$, the other two choices have $s\neq 0$ and thus determine $\mu$ by $(s,t)=\left(1,\frac{t}{s}\right)=(1,-\mu)$.

In order to keep our algorithm simple, we take as the first option $\rho x + \sigma y + \tau z=y$.
 Let $q_2(x,y,z)=b_{00}z^2+b_{10}xz+\cdots+b_{01}y^2$. It follows that
$m_2$ is decomposable if and only if
\begin{equation}\label{eq:demu}
\det\left(
\left[\begin{matrix}
2b_{20} & b_{11} & b_{10}\cr
b_{11} & 2b_{02} & b_{01}\cr
b_{10} & b_{01} & 2b_{00}\end{matrix}\right]
-\mu
\left[\begin{matrix}
0 & 1 & 0\cr
1 & -2\alpha_3 & -\beta_3\cr
0 & -\beta_3 & 0\end{matrix}\right]\right)=0.
\end{equation}
In the generic case \eqref{eq:demu} gives a quadratic equation for $\mu$ and
therefore has two solutions. We pick one and then apply Algorithm 1 to $m_2$.

%

However, it can happen that it is not possible to find such
$\mu$ that \eqref{eq:demu} holds.
This occurs if and only if $y(x-\alpha_3y-\beta_3z)$  is the only degenerate conic in
the pencil
\begin{equation}s\, q_2(x,y,z)+t\, y(x-\alpha_3y-\beta_3z).\label{eq:pencilmu}\end{equation}
By Lemma~\ref{lem:extrpencil} this implies that  $q_2$ is indecomposable and one of the lines $y=0$ or $x-\alpha_3 y-\beta_3z=0$ is tangent to $q_2$ at $(\beta_3,0,1)$.
In this case we can take $\rho x + \sigma y + \tau z=z$ and find $\mu'$ such that
\begin{equation}\label{eq:q2prime}
m'_2(x,y,z):=q_2(x,y,z)-\mu' z(x-\alpha_3y-\beta_3z)
\end{equation}
is decomposable, unless $z=0$ is also tangent to $q_2$ at $(\alpha_3,1,0)$. But if $q_2$ has the tangent $y=0$ at $(\beta_3,0,1)$ and the tangent $z=0$ at $(\alpha_3,1,0)$, then we can by Remark~\ref{rem:notline} pick $i\in\{1,2,4\}$ such that  $x-\alpha_i y-\beta_i z\neq x-\alpha_3 y-\beta_3z$.
Then one of the pencils
$$s\, q_2(x,y,z)+t\, y(x-\alpha_i y-\beta_i z)\ \mbox{ or }\ s\, q_2(x,y,z)+t\, z(x-\alpha_i y-\beta_i z)$$
contains more than one degenerate conic. We interchange the 3rd and $i$th diagonal element in the ansatz~\eqref{eq:rep4b} accordingly.


%
%

\section{Quintic  polynomials}\label{sec:quintic}

Let $p_5$ be a quintic bivariate polynomial in the homogeneous
form
$$p_5(x,y,z)=a_{00}z^5+a_{10}xz^4+a_{01}yz^4+\cdots+a_{50}x^5+\cdots+a_{05}y^5$$
that defines a quintic curve
$$\mathcal{C}=\left\{ (x,y,z)\in\CC\PP^2 \, : \, p_5(x,y,z)=0 \right\}.$$
As before we can assume that $a_{50}\ne 0$, which geometrically means that  $(1,0,0)\notin \mathcal{C}$.
Then each of the lines $\mathcal{L}_z=\left\{z=0  \right\}$ and $\mathcal{L}_y=\left\{y=0  \right\}$ intersects $\mathcal{C}$ in five points
$$ \mathcal{C}\cap \mathcal{L}_z=\left\{ (\alpha_i,1,0) \right\}_{i=1,2,3,4,5}\ \mbox{ and }\  \mathcal{C}\cap \mathcal{L}_y=\left\{ (\beta_j,0,1) \right\}_{j=1,2,3,4,5},$$
where  $\alpha_1,\ldots,\alpha_5$ are the roots of
$p_5(\alpha,1,0)=0$ and $\beta_1,\ldots,\beta_5$ are the roots of
$p_5(\beta,0,1)=0.$
By the reduction~\eqref{eq:qn} in  Lemma \ref{lem:reduk1}  there exists a homogeneous polynomial $q_3$ such that
\begin{equation}\label{eq:q6}
p_5(x,y,z)-a_{50}\prod_{j=1}^5(x-\alpha_j y -\beta_j z)=yz\,q_3(x,y,z).
\end{equation}
As always we exclude the case when $\mathcal{C}$ is a line. The following lemma implements Remark~\ref{rem:notline} for a quintic.
\begin{lemma}\label{lem:34}
If $p_5(x,y,z)\ne \left(\root 5\of {a_{50}}\, x + \root 5\of {a_{05}}\,
y + \root 5\of {a_{00}}\, z\right)^5$, then
a preliminary generic rotation of coordinates
$y$ and $z$ around $x$
\begin{equation}\label{eq:substit2}
\left[\begin{matrix}
x \cr
y  \cr
 z
\end{matrix}\right]=\left[\begin{matrix}
 1 & 0 &0 \cr
 0  & c & s \cr
 0 &-s &c
\end{matrix}\right]\, \left[\begin{matrix}
\widetilde{x} \cr
 \widetilde{y}  \cr
 \widetilde{z}
\end{matrix}\right]
\end{equation}
transforms the polynomial $p_5$ in such way that we can assume
that there exists
a permutation of
$\alpha_1,\ldots,\alpha_5$ and of $\beta_1,\ldots,\beta_5$
 such that $\alpha_3\ne\alpha_4$, $\beta_3\ne \beta_4$, and the intersection
$$\mathcal{L}_3=\left\{ x-\alpha_3y-\beta_3 z=0\right\}\, \cap \, \mathcal{L}_4=\left\{ x-\alpha_4y-\beta_4 z=0\right\}
\notin \mathcal{C}.$$
\end{lemma}

\begin{proof}
We know from $p_5(x,y,z)\ne \left(\root 5\of {a_{50}}\, x + \root 5\of {a_{05}}\,
y + \root 5\of {a_{00}}\, z\right)^5$ that $\mathcal{C}$ is not a line.
It follows that for a generic $\varphi$ each of the lines $\cos\varphi y + \sin\varphi z =0$
and $-\sin\varphi y + \cos\varphi z =0$
intersects ${\cal C}$ in at least two distinct
points. Denote by $T_1\ne T_2$ and $T_3\ne T_4$ the intersections of
$\cos\varphi y + \sin\varphi z =0$ and $-\sin\varphi y + \cos\varphi z =0$  with ${\cal C}$
respectively. Moreover, we can assume that at least one of the
intersections
$\mathcal{L}(T_1,T_3)\cap \mathcal{L}(T_2,T_4)$ or
$\mathcal{L}(T_1,T_4)\cap \mathcal{L}(T_2,T_3)$,
where $\mathcal{L}(T_i,T_j)$
is a line through $T_i$ and $T_j$,
does not lie on $\mathcal{C}$.

 Therefore, if we apply a preliminary transformation of coordinates
  \eqref{eq:substit2} where we take $c=\cos\varphi$ and $s=\sin\varphi$, then in the new coordinates
$\mathcal{L}_y$ and $\mathcal{L}_z$ intersect in $(1,0,0)\notin \mathcal{C}$ and
each of them intersects $\mathcal{C}$ in at least two distinct points.
We can thus permute $\alpha_1,\ldots,\alpha_5$ and
$\beta_1,\ldots,\beta_5$ so that $\alpha_3\ne\alpha_4$ and $\beta_3\ne \beta_4$.
\end{proof}
\medskip

\begin{figure}
\begin{center}
\includegraphics[width=8cm]{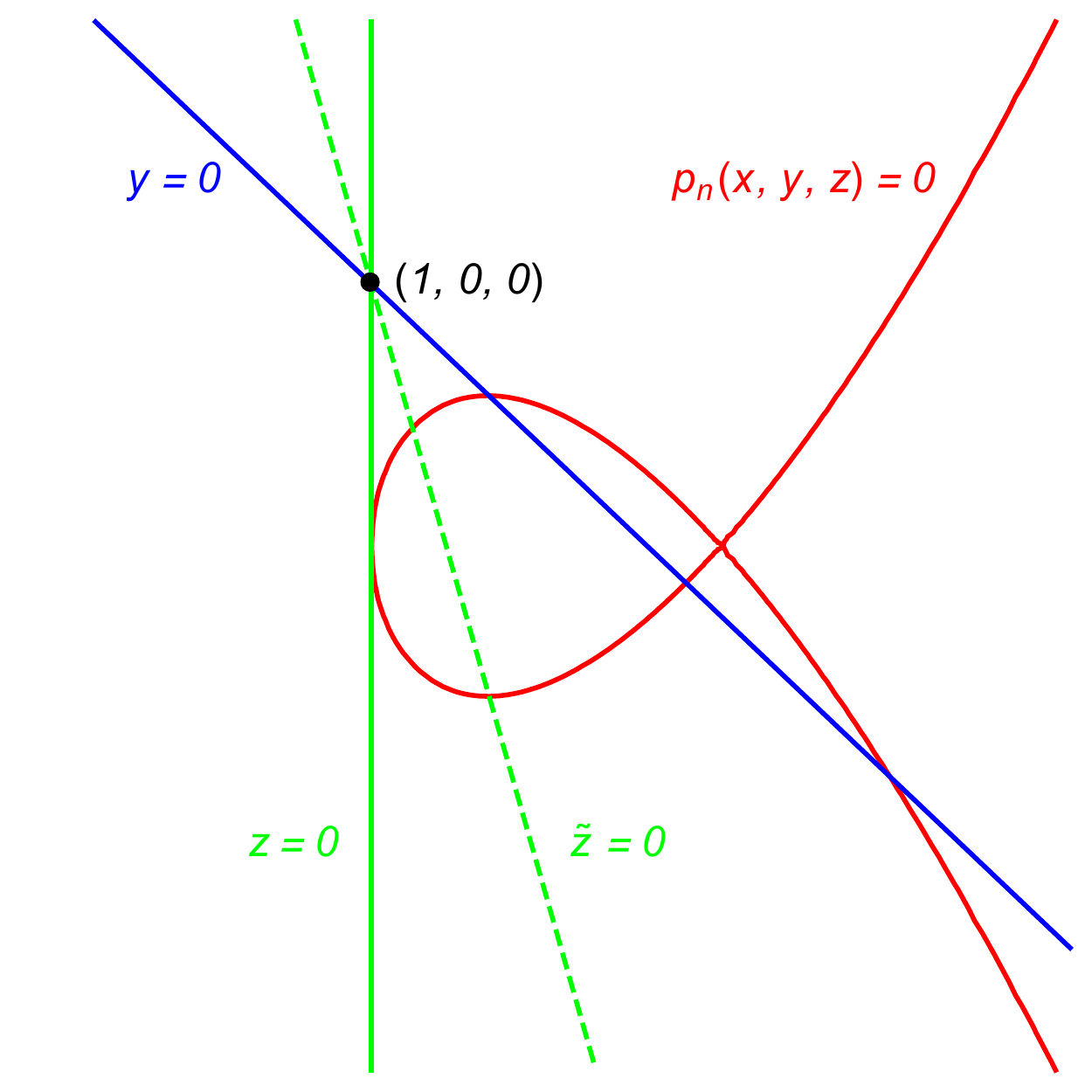}
\end{center}
\caption{Choose such coordinates that $\mathcal{L}_y$ and $\mathcal{L}_z$ intersect
$\mathcal{C}$ in more than one point.\label{BorPict2}
}
\end{figure}
\begin{figure}
\begin{center}
\includegraphics[width=8cm]{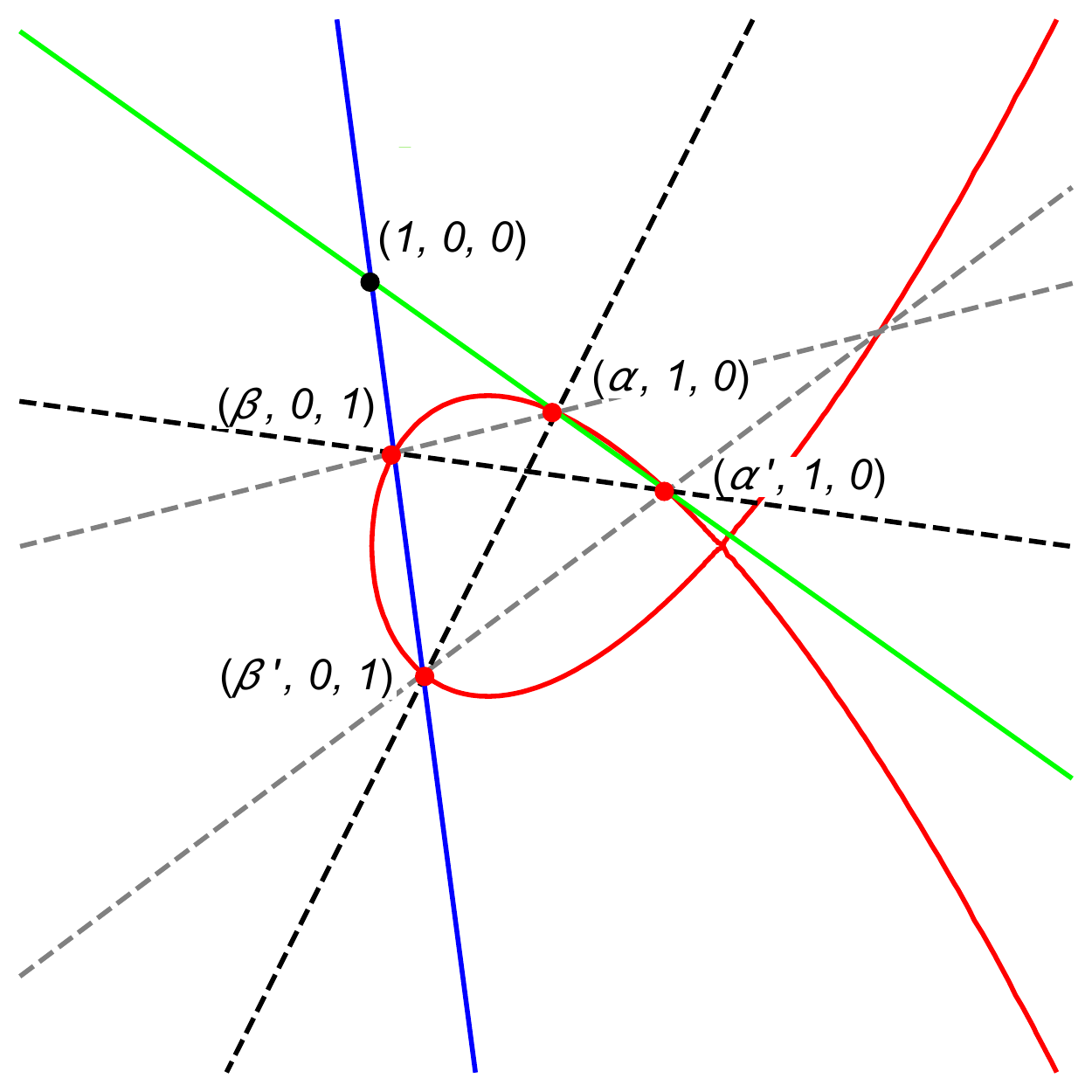}
\end{center}
\caption{The pencil of quadrics through $(\alpha,1,0),(\alpha',1,0), (\beta,0,1), (\beta',0,1)$ contains three singular quadrics: $ \left\{yz=0\right\},\, \left\{ (x-\alpha y-\beta z)(x-\alpha' y-\beta' z)=0\right\}$ and
 $\left\{ (x-\alpha y-\beta' z)(x-\alpha' y-\beta z)=0\right\}$.\label{BorPict3}}
\end{figure}

The ansatz for a determinantal
representation of $p_5$ is
\begin{equation}\label{eq:rep5b}
\left[\begin{matrix}
a_{50}(x-\alpha_1y-\beta_1 z) \!&\! y \!&\!  0 \!&\!  0 \!&\!  0 \cr
0 \!&\! x-\alpha_2y-\beta_2 z \!&\!  \gamma_1 x + \delta_1 y + \epsilon_1 z \!&\!  0 \!&\!   \gamma_4 x + \delta_4 y + \epsilon_4 z \cr
0 \!&\! 0 \!&\!  x-\alpha_3y-\beta_3 z \!&\!  \gamma_2 x + \delta_2 y + \epsilon_2 z \!&\!  0\cr
0 \!&\!0 \!&\!  0 \!&\!  x-\alpha_4y-\beta_4 z \!&\!   \gamma_3 x + \delta_3 y + \epsilon_3 z \cr
z \!&\! 0 \!&\!  0 \!&\!  0 \!&\!   x-\alpha_5 y-\beta_5 z
\end{matrix}\right],
\end{equation}
 whose determinant is
$$a_{50}\prod_{j=1}^5(x-\alpha_j y-\beta_j z)+y\,z\, \det
\left[\begin{matrix}
 \gamma_1 x + \delta_1 y + \epsilon_1 z & 0 &  \gamma_4 x + \delta_4 y + \epsilon_4 z \cr
 x-\alpha_3y-\beta_3 z & \gamma_2 x + \delta_2 y + \epsilon_2 z & 0\cr
0 & x-\alpha_4y-\beta_4 z &  \gamma_3 x + \delta_3 y + \epsilon_3 z
\end{matrix}\right].$$
Next we prove that there exist such $ \{\gamma_i, \delta_i, \epsilon_i\}_{i=1,2,3,4}$  that the above $3\times3$ determinant equals $q_3(x,y,z)$ in~\eqref{eq:q6}.

If $p_5(x,y,z)\ne \left(\root 5\of {a_{50}}\, x + \root 5\of {a_{05}}\,
y + \root 5\of {a_{00}}\, z\right)^5$ we can apply Lemma~\ref{lem:34} and permute the indices in $\alpha_1,\alpha_2,\alpha_3,\alpha_4,\alpha_5$ and 
$\beta_1,\beta_2,\beta_3,\beta_4,\beta_5$  so that the intersection
$$\mathcal{L}_3=\left\{ x-\alpha_3y-\beta_3 z=0\right\}\, \cap \, \mathcal{L}_4=\left\{ x-\alpha_4y-\beta_4 z=0\right\}
\notin \mathcal{C}.$$
Since
$$\mathcal{L}_3 \cap  \mathcal{L}_4=\left\{ x-\alpha_3y-\beta_3 z=0\right\}\cap \left\{ x-\alpha_4y-\beta_4 z=0\right\}=
\left(\alpha_3\beta_4-\alpha_4\beta_3,\beta_4-\beta_3,\alpha_3-\alpha_4
 \right)$$
 does not lie on $\mathcal{C}$, it is also not a zero of $q_3(x,y,z)$. Next we change the variables
 \begin{equation}\label{eq:change5}
\left[\begin{matrix}
 \widetilde{x} \cr
 \widetilde{y}  \cr
  \widetilde{z}
\end{matrix}\right]=\left[\begin{matrix}
 \alpha_3\beta_4-\alpha_4\beta_3 & \beta_4-\beta_3 & \alpha_3-\alpha_4 \cr
 1 & -\alpha_3 & -\beta_3  \cr
 1 & -\alpha_4 &-\beta_4
\end{matrix}\right]\, \left[\begin{matrix}
x \cr
y  \cr
z
\end{matrix}\right].
\end{equation}
Since the first row of the matrix in \eqref{eq:change5} is the cross product of
the second and third row, the matrix is invertible for
$(\alpha_3,\beta_3)\ne (\alpha_4,\beta_4)$.

Note that in the new coordinates $\mathcal{L}_3,  \mathcal{L}_4$ have
equations $\{\widetilde{y}=0\},\{\widetilde{z}=0\}$ respectively and the intersection point
$\mathcal{L}_3 \cap  \mathcal{L}_4$ becomes $(1,0,0)$. This implies
that $\widetilde{b}_{30}\neq 0$, where $\widetilde{b}_{30}$ is the coefficient of $\widetilde{q}_3(\widetilde{x},\widetilde{y},\widetilde{z})$ at $\widetilde x^3$.
Then we can perform the reduction~(\ref{eq:qn}) for $\widetilde{q}_3$ in the new variables
$$\widetilde{q}_3(\widetilde{x},\widetilde{y},\widetilde{z})-
\widetilde{b}_{30}\prod_{j=1}^3 (\widetilde{x}-\widetilde{\alpha}_j \widetilde{y} -\widetilde{\beta}_j \widetilde{z})=\widetilde{y}\,\widetilde{z}\, (\widetilde{\beta}_{10} \widetilde{x}+\widetilde{\beta}_{01} \widetilde{y}+\widetilde{\beta}_{00} \widetilde{z}),$$
which yields
$$\widetilde{q}_3(\widetilde{x},\widetilde{y},\widetilde{z})=\det
\left[\begin{matrix} \widetilde{a}_{30}(\widetilde{x}-\widetilde{\alpha}_1 \widetilde{y} -\widetilde{\beta}_1 \widetilde{z}) & 0 & \widetilde{\beta}_{10} \widetilde{x}+\widetilde{\beta}_{01} \widetilde{y}+\widetilde{\beta}_{00} \widetilde{z} \cr
 \widetilde{y} & \widetilde{x}-\widetilde{\alpha}_2 \widetilde{y} -\widetilde{\beta}_2 \widetilde{z} & 0 \cr
0 &   \widetilde{z} & \widetilde{x}-\widetilde{\alpha}_3 \widetilde{y} -\widetilde{\beta}_3 \widetilde{z} \end{matrix}\right].$$

Substituting back $x,y,z$ we obtain $\gamma_i,\delta_i,\epsilon_i$ for $i=1,2,3,4$
and
\begin{equation}\label{eq:repq3}
q_3(x,y,z)= \det
\left[\begin{matrix}
 \gamma_1 x + \delta_1 y + \epsilon_1 z & 0 &  \gamma_4 x + \delta_4 y + \epsilon_4 z \cr
 x-\alpha_3y-\beta_3 z & \gamma_2 x + \delta_2 y + \epsilon_2 z & 0\cr
0 & x-\alpha_4y-\beta_4 z &  \gamma_3 x + \delta_3 y + \epsilon_3 z
\end{matrix}\right].\end{equation}
\bigskip

In Lemma~\ref{lem:34} we showed that using a generic rotation
we can arrange the roots of $p_5(\alpha,1,0)=0$ and
the roots of $p_5(\beta,0,1)=0$ so that
that $\alpha_3\ne\alpha_4$ and $\beta_3\ne \beta_4$. Moreover, the intersection of the lines
$x-\alpha_3y-\beta_3 z=0$ and $x-\alpha_4y-\beta_4 z=0$ does not belong to $\mathcal{C}$.
This implies that after the change of variables we have $\widetilde b_{30}\ne 0$ in $\widetilde q_3$
and we can thus apply the procedure from Section \ref{sec:cubic}.

Let us remark that a preliminary change of coordinates from Lemma~\ref{lem:34}, where we require
 $\alpha_3\ne\alpha_4$ and $\beta_3\ne \beta_4$, is not necessary for our procedure. Based on Lemma \ref{lem:tang}, we can also find a determinantal representation in
the situation where $(\alpha_3,\beta_3)\ne(\alpha_4,\beta_4)$ and the
intersection $T=\{x-\alpha_3y-\beta_3 z=0\} \cap \{x-\alpha_4y-\beta_4 z=0\}$ is an element of $\mathcal{C}$.
In this case $T$ needs to be a smooth point and none of the lines $x-\alpha_3y-\beta_3 z=0$ and
$x-\alpha_4y-\beta_4 z=0$ should be a tangent to $\mathcal{C}$ at $T$. This ensures that
we can apply Lemma \ref{lem:tang} to get a determinantal representation of the
form \eqref{eq:repq3} for $q_3$.

In the implementation of our procedure it is useful to check in advance
if such conditions are fulfilled, in which case a preliminary change of variables does not need to be applied.

\section{Sextic  polynomials}\label{sec:sextic}
Let $p_n$ be a bivariate polynomial of degree $3\le n\le 5$.
Recall the shapes of the $4\times 4$ and $5 \times 5$ determinantal representations~\eqref{eq:rep4b} and~\eqref{eq:rep5b}, respectively, and observe that the
$3 \times 3$ determinant~\eqref{eq:rep3b} is the same as
$$\det \left[\begin{matrix}
a_{30}(x-\alpha_1y-\beta_1 z) & y & 0 \cr
0 & x-\alpha_2y-\beta_2 z &  b_{00}z + b_{10}x+b_{01}y \cr
z &  0 & x-\alpha_3y-\beta_3 z\end{matrix}\right].$$
 If $M=xA+yB+zC$ is a determinantal representations of $p_n$, then $M$ is of the following shape. 
With the exception of $z$ in the left lower corner, $M$ is
upper triangular with the diagonal elements
$a_{n0}(x-\alpha_1 y-\beta_1 z),
x-\alpha_2 y-\beta_2 z,\ldots,x-\alpha_n y-\beta_n z$,
where $\alpha_i$ and $\beta_i$ are the roots of
$p_n(\alpha,1,0)=0$ and
$p_n(\beta,0,1)=0$.
The first row of $M$ is $
[ a_{n0}(x-\alpha_1y-\beta_1 z), \, (-1)^{n-1} y, \, 0, \, \cdots, \,  0]$ and the submatrix $M(2\colon\! n-1,3\colon\! n)$ is a
determinantal representation of the polynomial $q_{n-2}$ from the reduction \eqref{eq:qn}.


Now, let $p_6$ be a sextic bivariate polynomial in the homogeneous form,
with $a_{60}=p_6(1,0,0)\ne 0$. If we try to extend the methods from
the previous sections, then the first step is to
apply Lemma~\ref{lem:reduk1}
to obtain the reduction
\begin{equation}\label{eq:sextic}
p_6(x,y,z)-a_{60}\prod_{j=1}^6(x-\alpha_j y -\beta_j z)=yz\,q_4(x,y,z),
\end{equation}
where $q_4$ is a homogeneous polynomial of degree $4$.
Following the same approach as for $n\leq5$, an appropriate shape for a
determinantal representation of $p_6$ seems to be
\begin{equation}\label{eq:sextictry} {\small
\left[\begin{matrix}a_{60}(x-\alpha_1y-\beta_1 z)\hspace{-1em} \!&\! -y \!&\! 0 \!&\! 0 \!&\! 0 \!&\! 0\cr
0 \!&\! x-\alpha_2y-\beta_2 z \!&\!* \!&\! * \!&\! * \!&\!
* \cr
0 \!&\! 0 \!&\! x-\alpha_3y-\beta_3 z \!&\! * \!&\! * \!&\! * \cr
0 \!&\! 0 \!&\! 0 \!&\! x-\alpha_4y-\beta_4 z \!&\! * \!&\! *\cr
0 \!&\! 0 \!&\! 0 \!&\! 0 \!&\!  x-\alpha_5 y-\beta_5 z & *  \cr
z \!&\! 0 \!&\! 0 \!&\! 0 \!&\! 0 \!&\! x-\alpha_6 y-\beta_6 z\end{matrix}\right]}.
\end{equation}
Here $*$ denote possibly nonzero elements of the
form $\rho x + \sigma y + \tau z $ such that
\eqref{eq:sextictry} is a determinantal representation of $p_6$.
This holds if
$$\det\left[\begin{matrix}
* \!&\! * \!&\! * \!&\! * \cr
x-\alpha_3y-\beta_3 z \!&\! * \!&\! * \!&\! * \cr
0 \!&\! x-\alpha_4y-\beta_4 z \!&\! * \!&\! *\cr
0 \!&\! 0 \!&\!  x-\alpha_5 y-\beta_5 z & *
\end{matrix}\right]=q_4(x,y,z).$$
The above is a $4\times 4$ upper Hessenberg matrix whose three elements
on the subdiagonal are fixed. While we were able to derive simple algorithms for $n=4$ and $n=5$, where we have submatrices of
size $2\times 2$  and $3\times 3$ such that one and two elements are
fixed, respectively,  at present we have no practical
algorithm for the case $n=6$. The main obstacle is that it is not possible to apply the reduction from  Lemma~\ref{lem:reduk1}
to the determinant of the $4\times 4$ submatrix with three fixed elements.

This does not imply that a $6\times 6$ representation for
a sextic polynomial does not exist. We know from \cite{Dixon} that a representation of the minimum size always exists, but a different construction needs to be applied. For instance, the construction from \cite{Bornxn}
gives an $n\times n$ representation for a square-free bivariate polynomial
of degree $n$, i.e., a polynomial that is not a multiple of a square of a
non-constant polynomial.

\section{Algorithm}\label{sec:numalg}

The following algorithm encapsulates the results from the previous sections. We can
apply it to construct a determinantal representation of a bivariate polynomial of small degree in the
homogeneous form.


\noindent\rule{\textwidth}{1pt}
\noindent {\bf Algorithm 2.} Given a bivariate polynomial $p_n$ of degree $2\le n\le 5$
in the homogeneous form
$$
p(x,y,z)=a_{00}z^n+a_{10}xz^{n-1}+a_{01}yz^{n-1}+\cdots+a_{n0}x^n+\cdots+a_{0n}y^n,
$$
where at least one of the coefficients $a_{n0},a_{n-1,1},\ldots,a_{0n}$ is nonzero,
the algorithm returns $n\times n$ matrices $A$, $B$, and $C$, such that
$\det(xA+yB+zC)=p(x,y,z)$.\\[-0.5em]
\noindent\rule{\textwidth}{0.5pt}

\begin{enumerate}
\item[(1)]If $n=5$, test if
$p_5$ has the form $p_5(x,y,z)=(\alpha x +\beta y + \gamma z)^5$. Compute
the residual
$$r_5(x,y,z)= p_5(x,y,z)-\left(\root 5\of {a_{50}}\, x + \root 5\of {a_{05}}\,
y + \root 5\of {a_{00}}\, z\right)^5$$
and, if $r_5\equiv 0$, 
return
$A=\root 5\of {a_{50}}\, I$,
$B=\root 5\of {a_{05}}\, I$, and $C=\root 5\of {a_{00}}\, I$.
\item[(2)] If $a_{n0}=0$, apply a linear substitution of variables
$x=c_1 \widetilde x+s_1 \widetilde y$, $y=-s_1 \widetilde x + c_1 \widetilde y$,
$z=\widetilde z$, where  $c_1$ and $s_1$ are selected such that $c_1^2+s_1^2=1$ and $p_n(c_1,s_1,0)\ne 0$.
\item[(3)] Compute the roots $\alpha_1,\alpha_2,\ldots,\alpha_n$ of
$p_n(\alpha,1,0)=a_{n0} \alpha^n+a_{n-1,1}\alpha^{n-1}  + \cdots +a_{0n}=0$
and the roots $\beta_1,\beta_2,\ldots,\beta_n$ of
$p_n(\beta,0,1)=a_{00}+a_{10}\beta + \cdots + a_{n0}\beta^n=0.$

If $n=5$, order the roots 
so that $\alpha_3\ne \alpha_4$, $\beta_3\ne \beta_4$,
and the intersection of
$x-\alpha_3y-\beta_3 z=0$ and $x-\alpha_4y-\beta_4 z=0$ does not
lie on $p_5(x,y,z)=0$. If this is not possible,
apply a linear substitution
$y=c_2 \widetilde y+s_2 \widetilde z$, $z=-s_2 \widetilde y + c_2 \widetilde z$,
$x=\widetilde x$, where random $c_2$ and $s_2$ are selected such that $c_2^2+s_2^2=1$, and
return to step (3).
\item[(4)] Compute the polynomial $q_{n-2}(x,y,z)
=r_{00}z^{n-2}+\cdots+r_{n-2,0}x^{n-2}+\cdots+r_{0,n-2}y^{n-2}$
 of degree $n-2$ such that
$$p_n(x,y,z)-a_{n0}\prod_{i=1}^n(x-\alpha_iy-\beta_iz)=yz q_{n-2}(x,y,z).$$
\item[(5)] If $n=2$,
 set
$$xA+yB+zC=\left[\begin{matrix} a_{20}(x-\alpha_1y-\beta_1z) & -r_{00}y\cr
z & x-\alpha_2y-\beta_2z\end{matrix}\right].$$
\item[(6)] If $n=3$, 
set
$$xA+yB+zC=\left[\begin{matrix}a_{30}(x-\alpha_1y-\beta_1 z) & 0 & r_{10}x + r_{01}y+r_{00}z\cr
y & x-\alpha_2y-\beta_2 z & 0 \cr
0 &  z & x-\alpha_3y-\beta_3 z\end{matrix}\right].$$
\item[(7)] If $n=4$, then:
\begin{enumerate}
\item[a)]
Obtain $\widetilde q_2$ from $q_2$ by the change of variables
$x=\widetilde x + \alpha_3 \widetilde y+\beta_3\widetilde z$,
$y=\widetilde y$, and $z=\widetilde z$.
\item[b)]Apply the proof of Lemma \ref{lm:kvadrat} to get
$\widetilde l_i(\widetilde x,\widetilde y,\widetilde z)=\widetilde r_i \widetilde x + \widetilde s_i \widetilde y
+\widetilde t_i \widetilde z$ for $i=1,2,3$ such that
$$\widetilde q_2(\widetilde x,\widetilde y,\widetilde z)=
\widetilde l_1(\widetilde x,\widetilde y,\widetilde z) \widetilde l_2(\widetilde x,\widetilde y,\widetilde z)-
\widetilde x \widetilde l_3(\widetilde x,\widetilde y,\widetilde z).$$
\item[c)]Change the variables back to obtain
$l_i(x,y,z)=r_i x + s_i y + t_i z$ from $\widetilde l_i$ for $i=1,2,3$.
\end{enumerate}
\smallskip
Set $xA+yB+zC=$
$$
\left[\begin{matrix}a_{40}(x-\alpha_1y-\beta_1 z) & -y & 0 & 0\cr
0 & x-\alpha_2y-\beta_2 z & r_1 x + s_1 y + t_1 z & r_3 x + s_3 y + t_3 z\cr
0 & 0 & x-\alpha_3y-\beta_3 z & r_2 x + s_2 y + t_2 z \cr
z & 0 & 0 & x-\alpha_4y-\beta_4 z\end{matrix}\right].$$
%
%
 \item[(8)] If $n=5$, then:
\begin{enumerate}
\item[a)] Apply the change of variables
$$
\left[\begin{matrix}
 \widetilde{x} \cr
 \widetilde{y}  \cr
  \widetilde{z}
\end{matrix}\right]=\left[\begin{matrix}
 \alpha_3\beta_4-\alpha_4\beta_3 & \beta_4-\beta_3 & \alpha_3-\alpha_4 \cr
 1 & -\alpha_3 & -\beta_3  \cr
 1 & -\alpha_4 &-\beta_4
\end{matrix}\right]\, \left[\begin{matrix}
x \cr
y  \cr
z
\end{matrix}\right]$$
to obtain $\widetilde q_3$ from $q_3$.
\item[b)] Apply the algorithm recursively on $\widetilde q_3$ to obtain
a representation of the form
$$\widetilde{q}_3(\widetilde{x},\widetilde{y},\widetilde{z})=\det\left(
\left[\begin{matrix} \widetilde{q}_{30}(\widetilde{x}-\widetilde{\alpha}_1 \widetilde{y} -\widetilde{\beta}_1 \widetilde{z}) & 0 & \widetilde{r}_{10} \widetilde{x}+\widetilde{r}_{01} \widetilde{y}+\widetilde{r}_{00} \widetilde{z} \cr
 \widetilde{y} & \widetilde{x}-\widetilde{\alpha}_2 \widetilde{y} -\widetilde{\beta}_2 \widetilde{z} & 0 \cr
0 &   \widetilde{z} & \widetilde{x}-\widetilde{\alpha}_3 \widetilde{y} -\widetilde{\beta}_3 \widetilde{z} \end{matrix}\right]\right).$$
\item[c)]
Change the variables back to obtain $r_i,s_i,t_i$ for $i=1,2,3,4$
such that
$$q_3(x,y,z)= \det\left(
\left[\begin{matrix}
 r_1 x + s_1 y + t_1 z & 0 &  r_4 x + s_4 y + t_4 z \cr
 x-\alpha_3y-\beta_3 z & r_2 x + s_2 y + t_2 z & 0\cr
0 & x-\alpha_4y-\beta_4 z &  r_3 x + s_3 y + t_3 z
\end{matrix}\right]\right).$$
\end{enumerate}
Set $xA+yB+zC=$
$$\left[\begin{matrix}a_{50}(x-\alpha_1y-\beta_1 z)\hspace{-1em} & y & 0 & 0 &0 \cr
0 & x-\alpha_2y-\beta_2 z & r_1 x + s_1 y + t_1 z & 0 &  r_4 x + s_4 y + t_4 z \cr
0 & 0 & x-\alpha_3y-\beta_3 z & r_2 x + s_2 y + t_2 z & 0\cr
0 &0 &0 & x-\alpha_4y-\beta_4 z &  r_3 x + s_3 y + t_3 z \cr
z &0 &0 &0 &  x-\alpha_5 y-\beta_5 z
\end{matrix}\right].$$
\item[(9)]If a substitution was used in Step (2) or Step (5), substitute the variables back
before returning the final determinantal representation $xA+yB+zC$.\\
 \end{enumerate}\vspace{-2em}
\noindent\rule{\textwidth}{0.5pt}
\medskip

Some comments:
\begin{itemize}
\item In Step (2) and Step (5) we apply a rotation
of variables $x,y$ around $z$ and of $y,z$ around $x$, respectively. Such
changes of variables are also used in \cite{Jonsson}.
\item After the substitution in Step (2) we should continue with the
 polynomial
$\widetilde p_n(\widetilde x,\widetilde y,\widetilde z)$ such
that $\widetilde a_{n0}\ne 0$. However, to keep the notation simple,
we again write $p_n(x,y,z)$ instead of $\widetilde p_n(\widetilde x,\widetilde y,\widetilde z)$ in Step (3) and further,
where we
assume now that $a_{n0}\ne 0$. If a change of variables was used,
we change back to the original variables in Step (9).
\item Even if $p_n$ is a polynomial
with real coefficients, the representation might
be complex
because the roots $\alpha_1,\ldots,\alpha_n$ and
$\beta_1,\ldots,\beta_n$ are not necessarily real.
\end{itemize}
\bigskip

\section{Numerical examples}\label{sec:numex}
The first example shows the output
of Algorithm 2 for a quintic bivariate polynomial.

\begin{example}\rm
We take the polynomial $$p(x,y,z)=x(x-y-z)(x+y+z)(x-2y-2z)(x+2y+2z)+yz^4+y^2z^3+y^3z^2.$$
If we order coefficients $\alpha_i$ and $\beta_i$ for $i=1,\ldots,5$ as
$\alpha_1=\beta_1=2$, $\alpha_2=\beta_2=-2$, $\alpha_3=\beta_3=1$, $\alpha_4=\beta_4=-1$, and
$\alpha_5=\beta_5=0$, then Algorithm 2 returns the determinantal representation
$$\left[\begin{matrix}x-2y-2 z\hspace{-1em} & y & 0 & 0 &0 \cr
0 & x+2y+2 z & \frac{1}{32}(i\sqrt{3} x +y +z) & 0 &  -\frac{3}{4}z \cr
0 & 0 & x-y-z & 2(i\sqrt{3}x-y+z) & 0\cr
0 &0 &0 & x+y+ z & 4 z \cr
z &0 &0 &0 &  x
\end{matrix}\right].$$
\end{example}
\medskip

Although Algorithm 2 works well in the exact computation,  we introduced some
modifications in the numerical implementation in order to make it more numerically stable. Some
of them are:
\begin{itemize}
\item
Instead of using rotations
of coordinates $x,y$ around $z$ in Step (2) and $y,z$ around $x$ in Step (5),
we rather apply a transformation (\ref{eq:transt}), where $T$
is in both cases a random orthogonal $3\times 3$ matrix. This prevents that
$|a_{n0}|$ is small
 compared to $\max_{i+j\le n}|a_{ij}|$, as then some of the
 roots $\alpha_1,\ldots,\alpha_n$ and
$\beta_1,\ldots,\beta_n$ might have large absolute values and the
matrix in (\ref{eq:change5}) might be ill-conditioned.
\item For $n=4$ we order the roots so that
$|\alpha_3|=\min_{i=1,\ldots,4}|\alpha_i|$
and
$|\beta_3|=\min_{i=1,\ldots,4}|\beta_i|$
to minimize the condition number
of the change of variables in Step (7a).
\end{itemize}

More details can be found in the implementation of Algorithm 2 in Matlab \cite{Matlab}, which is included in \cite{BorBR}. We
applied Algorithm 2 to numerically
solve
random systems of bivariate polynomials of small degrees by
the  approach proposed in \cite{BorMichiel}. We denote this method
by {\tt Lin345}. The main idea is
to treat the system as a two-parameter
eigenvalue problem using determinantal representations.

We start with a system of two bivariate polynomials
\begin{equation}\label{eq:p}
\begin{split}
p(x,y) &:= \sum_{i=0}^{n_1} \, \sum_{j=0}^{{n_1}-i} \ p_{ij} \, x^i \, y^j=0, \\
q(x,y) &:= \sum_{i=0}^{n_2} \, \sum_{j=0}^{{n_2}-i} \ q_{ij} \, x^i \, y^j=0.
\end{split}
\end{equation}
and use Algorithm 2 to compute  matrices
$A_1,B_1,C_1$ and $A_2, B_2, C_2$ such that
\begin{equation}\label{eq:dr}
\begin{split}
\det(A_1+x B_1 +y C_1)&=p(x,y),\\[1mm]
\det(A_2+x B_2 +y C_2)&=q(x,y).
\end{split}
\end{equation}
A root $(x,y)$ of \eqref{eq:p} corresponds to an
eigenvalue of the two-parameter eigenvalue problem \cite{Atkinson}
\begin{equation}\label{eq:twopar}
\begin{matrix}
(A_1+x B_1+y C_1) \, u=0,\\[0.2em]
(A_2+x B_2+y C_2) \, v=0,
\end{matrix}
\end{equation}
where $u$ and $v$ are nonzero vectors.
See \cite{BorMichiel} and references therein for details
on the two-parameter eigenvalue problems and the available
numerical methods. To
solve \eqref{eq:twopar} we
consider a pair
 of generalized eigenvalue problems
\begin{equation}\label{eq:twopardelta}
\begin{matrix}
(\Delta_1-x \Delta_0) \, w=0,\\[0.2em]
(\Delta_2-y \Delta_0) \, w=0,
\end{matrix}
\end{equation}
where
$\Delta_0=B_1\otimes C_2-C_1\otimes B_2$,
$\Delta_1=C_1\otimes A_2-A_1\otimes C_2$,
$\Delta_2=A_1\otimes B_2-B_1\otimes A_2$
and $w=u\otimes v$.

\medskip

\begin{example}\rm\label{ex:genpl}
In this example we generated random bivariate polynomials whose coefficients are random real
numbers uniformly distributed on $[0,1]$ or random complex numbers, such that real
and imaginary parts are both uniformly distributed on $[0,1]$. We compared
{\tt Lin345} to  {\tt Lin2} from \cite{BorMichiel}, which returns matrices of size $3$, $5$ and $8$ for a generic bivariate polynomial or degree $3$, $4$, and $5$, respectively, and to {\tt MinRep} from \cite{Bornxn}, which returns matrices of the same size as the degree of a square-free polynomial. These are the only two methods that we compared {\tt Lin345} to, since other methods for solving systems of bivariate polynomials (for example \cite{Bornxn} and   \cite{BorMichiel}) return representations of bigger sizes  and moreover turn out to be slower.

For each $n$ we tested the three methods on 500 systems with real and 500 systems with complex polynomials. We measured the average computational time and
the accuracy of the obtained solutions. A measure of accuracy is the
 maximum value of
\begin{equation}
\max_{i=1,\ldots,n^2}\left(\max(|p_1(x_i,y_i)|,|p_2(x_i,y_i)|) \ \|J^{-1}(x_i,y_i)\|^{-1}\right),\label{eq:condn}
\end{equation}
where $J(x_i,y_i)$ is the Jacobian matrix of $p_1$ and $p_2$ at the computed root $(x_i,y_i)$. Here
$\|J^{-1}(x_i,y_i)\|^{-1}$ is an absolute
condition number of the root $(x_i,y_i)$ and we assume that in random
examples all roots are simple. The results in Table \ref{tbl:prim} show that for generic polynomials of degrees $3$ to $5$
{\tt Lin345} is faster and as accurate as {\tt Lin2} and {\tt MinRep.}

\begin{table}[!htbp]
\begin{footnotesize}
\begin{center}
\caption{Average computational time (arithmetic mean) in milliseconds
and average accuracy (geometric mean)
of {\tt Lin345}, {\tt Lin2}, and {\tt MinRep} for
random full bivariate polynomial systems of degrees $3$ to $5$.
}\label{tbl:prim}
\begin{tabular}{c|ccc|ccc} \hline \rule{0pt}{2.3ex}%
& \multicolumn{3}{|c|}{average time in ms} &
\multicolumn{3}{|c}{average accuracy} \\
degree & {\tt Lin345} & {\tt Lin2}  & {\tt MinRep}
& {\tt Lin345} & {\tt Lin2}  & {\tt MinRep}\\
\hline \rule{0pt}{2.3ex}%
3 & 1.3 & 2.0 & 3.6 &
    $8.0\cdot 10^{-15}$ & $6.9\cdot 10^{-15}$ & $1.1\cdot 10^{-14}$ \\
4 & 2.6 & 4.5 & 4.8 &
    $3.6\cdot 10^{-14}$ & $4.9\cdot 10^{-14}$ & $5.5\cdot 10^{-14}$ \\
5 & 4.8 & 9.8 & 6.6 &
    $2.1\cdot 10^{-13}$ & $6.2\cdot 10^{-14}$ & $3.7\cdot 10^{-13}$ \\
\hline
\end{tabular}
\end{center}
\end{footnotesize}
\end{table}

\end{example}
\medskip

\begin{example}\rm\label{ex:gensq}
In the second example we generate one of the polynomials in the same way as
in Example \ref{ex:genpl}, while we generate the other as
\[p_2(x,y)=(\alpha x+\beta y +\gamma)^2 q_2(x,y),\]
where $\alpha,\beta,\gamma$ are random numbers and
$q_2(x,y)$ is a random polynomial of degree $n-2$ for $n=3,4,5$. Since the
second polynomial is not square-free, we cannot apply {\tt MinRep}.
This however is not an obstacle for {\tt Lin345} that computes determinantal representations with $n\times n$ matrices.

\begin{table}[!htbp]
\begin{footnotesize}
\begin{center}
\caption{Average computational time (arithmetic mean) in milliseconds
and average accuracy (geometric mean)
of {\tt Lin345} and {\tt Lin2} for
random full bivariate polynomial systems of degrees $3$ to $5$
such that one polynomial is a multiple of a square of a linear polynomial.
}\label{tbl:square}
\begin{tabular}{c|cc|cc} \hline \rule{0pt}{2.3ex}%
& \multicolumn{2}{|c|}{average time in ms} &
\multicolumn{2}{|c}{average accuracy} \\
degree & {\tt Lin345} & {\tt Lin2}
& {\tt Lin345} & {\tt Lin2} \\
\hline \rule{0pt}{2.3ex}%
3 & 2.5 & 3.6  &
    $7.9\cdot 10^{-8}$ & $3.2\cdot 10^{-7}$ \\
4 & 4.0 & 6.7  &
    $7.7\cdot 10^{-8}$ & $1.6\cdot 10^{-7}$ \\
5 & 7.0 & 12.9 &
    $2.1\cdot 10^{-7}$ & $5.6\cdot 10^{-7}$ \\
\hline
\end{tabular}
\end{center}
\end{footnotesize}
\end{table}

For each $n$ we tested  {\tt Lin345} and   {\tt Lin2} on 500 systems with real and 500 systems with complex polynomials.
The results are presented in Table \ref{tbl:square}.
The computation takes longer than in Example \ref{ex:genpl}, because a slower
method needs to be applied to the two-parameter eigenvalue problem when multiple eigenvalues are detected. Moreover, the computed roots are not as accurate as in Example \ref{ex:genpl},
but this is expected as some of the roots are double and in numerical computations double roots behave as pairs of highly conditioned simple roots.
\end{example}
\medskip

In Example \ref{ex:permute} we showed that
permutations of $\alpha_i$ and of $\beta_i$ in Lemma~\ref{lem:reduk1} yield nonequivalent representations.
Next example shows that a change of variables can also result in nonequivalent determinantal representations.
\medskip

\begin{example}{\rm
Consider the Weierstrass cubic $p_3(x,y,z)=x(x+ y)(x-y)-y z^2=0$. The reduction~\eqref{eq:qn} on $\{y=0\},\ \{z=0\}$ yields the determinantal representation \eqref{eq:wei1} with
$\theta_1=1$ and $\theta_2=-1$ from Example~\ref{ex:Weier}.

To obtain another representation we apply the rotation~\eqref{eq:substit2} to vary the lines $\mathcal{L}_y$ and $\mathcal{L}_z$.
The rotation for $\pi/4$ around $x$ induces the following change of coordinates
$$\left[\begin{matrix}
x \cr
y  \cr
 z
\end{matrix}\right]=\left[\begin{matrix}
 1 & 0 &0 \cr
 0  & \sqrt{2}/2 & \sqrt{2}/2 \cr
 0 &-\sqrt{2}/2 & \sqrt{2}/2
\end{matrix}\right]
\, \left[\begin{matrix}
\widetilde{x} \cr
 \widetilde{y}  \cr
 \widetilde{z}
\end{matrix}\right].$$
By the reduction~\eqref{eq:qn} we obtain
$$\widetilde{p}_3(\widetilde{x},\widetilde{y},\widetilde{z})=\prod_{i=1}^3 \,(\widetilde{x}-\alpha_i\, \widetilde{y}-\beta_i\, \widetilde{z})
+\sqrt{2}\, \widetilde{y}\, \widetilde{z}\, (\widetilde{y}+\widetilde{z}), $$
where we choose
\begin{align*}
\alpha_1=\beta_1&=0.936717,\\
\alpha_2=\beta_2&=-0.468359 + 0.397592i,\\
\alpha_3=\beta_3&=-0.468359 - 0.397592i.
\end{align*}
Substituting $x,y,z$
back into the representation
$$\left[\begin{matrix} \widetilde{x}-\alpha_1 \widetilde{y}-\beta_1 \widetilde{z} & 0 & - \widetilde{z}\cr
 \widetilde{y} & \widetilde{x}-\alpha_2 \widetilde{y}-\beta_2 \widetilde{z} & 0 \cr
0 &   \widetilde{z} & \widetilde{x}-\alpha_3 \widetilde{y}-\beta_3 \widetilde{z} \end{matrix}\right]$$
gives
$$\left[\begin{matrix}x - 1.324717 y & 0 & 2y \cr
(y-z)/\sqrt{2} &  x + (0.662358 - 0.562279i) y & 0 \cr
0 &  (y+z)/\sqrt{2}  &  x + (0.662358 + 0.562279i) y  \end{matrix}\right].$$
}\end{example}
\\

\section{Conclusions}\label{sec:conc}

We presented a simple numerical algorithm for determinantal representations
of bivariate polynomials of degree $n\le 5$ with $n\times n$
matrices. Contrary to the other existing methods,
our algorithm works for arbitrary polynomials.
For the next degree, $n=6$, we did not succeed 
to
apply the same approach. The
smallest known determinantal representation that can be constructed
efficiently for any bivariate polynomial of degree $6$ thus remains to be of size
$10\times 10$ from~\cite{Bornxn} or \cite{BorMichiel}.

While the obtained representations have the optimal size according
to Dixon's theorem, they are not symmetric.
Let us remark that constructions
of symmetric representations are much more demanding as one needs to take into account additional geometry, for example flexes for cubics and bitangents for quartics (as explained for smooth curves in~\cite{HarrisGalois} and~\cite{VinnikovComplete}).
The reason is that a  smooth curve of degree $n$ has only a finite number of symmetric determinantal representations; on the other hand, all its determinantal representations can be parametrized by an open subset of the $\frac{(n-1)(n-2)}{2}$  dimensional Jacobian variety~\cite{VinnikovComplete}.

%
%
%

\end{document}